\documentclass[a4paper,10pt]{amsart}

\usepackage[utf8]{inputenc}
\usepackage[T1]{fontenc}
\usepackage{lmodern}
\usepackage{amsmath}
\usepackage{amsfonts}
\usepackage{amssymb}
\usepackage{amsthm}
\usepackage{thmtools}
\usepackage{mathtools}
\usepackage{enumitem}
\usepackage{placeins}
\usepackage{multirow}
\usepackage[pdftex,
            pdfauthor={Jethro van Ekeren, Ching Hung Lam, Sven Möller and Hiroki Shimakura},
            pdftitle={Schellekens' List and the Very Strange Formula},
            pdfsubject={},
            pdfkeywords={},
            pdfproducer={},
            pdfcreator={}]{hyperref}
\usepackage{caption}

\theoremstyle{plain}
\newtheorem{thm}{Theorem}[section]
\newtheorem{prop}[thm]{Proposition}

\newtheorem{lem}[thm]{Lemma}

\newtheorem*{thm*}{Theorem}
\newtheorem*{conj*}{Conjecture}
\newtheorem*{prop*}{Proposition}

\newtheoremstyle{narrow}
  {.5em} 
  {.5em} 
  {\itshape} 
  {} 
  {\bfseries} 
  {.} 
  {.5em} 
  {} 

\theoremstyle{narrow}

\newtheorem*{thn*}{Theorem}

\newtheorem*{conjn*}{Conjecture}

\theoremstyle{definition}
\newtheorem{defi}[thm]{Definition}

\newtheorem*{nota*}{Notation}
\newtheorem{rem}[thm]{Remark}

\newcommand{\Q}{\mathbb{Q}}
\newcommand{\Z}{\mathbb{Z}}
\newcommand{\Ns}{\mathbb{Z}_{>0}}
\newcommand{\N}{\mathbb{Z}_{\geq0}}
\newcommand{\C}{\mathbb{C}}
\newcommand{\R}{\mathbb{R}}
\renewcommand{\H}{\mathbb{H}}

\newcommand{\tr}{\operatorname{tr}}
\renewcommand{\i}{\mathrm{i}}
\newcommand{\e}{\mathrm{e}}

\newcommand{\Aut}{\operatorname{Aut}}
\newcommand{\Hom}{\operatorname{Hom}}

\newcommand{\rk}{\operatorname{rk}}

\newcommand{\voa}{vertex operator algebra}

\newcommand{\VOA}{Vertex Operator Algebra}
\newcommand{\vosa}{vertex operator subalgebra}

\newcommand{\fpvosa}{fixed-point vertex operator subalgebra}

\newcommand{\vac}{\textbf{1}}

\newcommand{\id}{\operatorname{id}}
\newcommand{\amgis}{\zeta}
\newcommand{\eps}{\varepsilon}
\newcommand{\lcm}{\operatorname{lcm}}

\newcommand{\ee}{\mathfrak{e}}
\newcommand{\g}{\mathfrak{g}}
\newcommand{\hh}{\mathfrak{h}}
\newcommand{\h}{{L_\C}}
\newcommand{\ad}{\operatorname{ad}}

\newcommand{\orb}{\operatorname{orb}}

\newcommand{\strathol}{strongly rational, holomorphic}
\newcommand{\strat}{strongly rational}

\renewcommand{\O}{\operatorname{O}}
\newcommand{\Co}{\operatorname{Co}}
\newcommand{\spn}{\operatorname{span}}
\newcommand{\gdh}{generalised deep hole}
\newcommand{\Gdh}{Generalised deep hole}
\newcommand{\GDH}{Generalised Deep Hole}
\newcommand{\Com}{\operatorname{Com}}

\begin{document}

\title{Schellekens' List and the Very Strange Formula}
\author[J.\ van Ekeren, C.\ H.\ Lam, S.\ Möller and H.\ Shimakura]{Jethro van Ekeren\textsuperscript{\lowercase{a}}, Ching Hung Lam\textsuperscript{\lowercase{b}},\\ Sven Möller\textsuperscript{\lowercase{c}} and Hiroki Shimakura\textsuperscript{\lowercase{d}}}
\thanks{\textsuperscript{a}{Universidade Federal Fluminense, Niterói, RJ, Brazil}}
\thanks{\textsuperscript{b}{Academia Sinica, Taipei, Taiwan}}
\thanks{\textsuperscript{c}{Rutgers University, Piscataway, NJ, United States of America}}
\thanks{\textsuperscript{d}{Tohoku University, Sendai, Japan}}
\thanks{Email: \href{mailto:jethrovanekeren@gmail.com}{\nolinkurl{jethrovanekeren@gmail.com}}, \href{mailto:chlam@math.sinica.edu.tw}{\nolinkurl{chlam@math.sinica.edu.tw}}, \href{mailto:math@moeller-sven.de}{\nolinkurl{math@moeller-sven.de}}, \href{mailto:shimakura@tohoku.ac.jp}{\nolinkurl{shimakura@tohoku.ac.jp}}}

\begin{abstract}
In \cite{Sch93} (see also \cite{EMS20a}) Schellekens proved that the weight-one space $V_1$ of a \strathol{} \voa{} $V$ of central charge $24$ must be one of $71$ Lie algebras. During the following three decades, in a combined effort by many authors, it was proved that each of these Lie algebras is realised by such a \voa{} and that, except for $V_1=\{0\}$, this \voa{} is uniquely determined by $V_1$. Uniform proofs of these statements were given in \cite{MS19,HM20}.

In this paper we give a fundamentally different, simpler proof of Schellekens' list of $71$ Lie algebras. Using the dimension formula in \cite{MS19} and Kac's ``very strange formula'' \cite{Kac90} we show that every \strathol{} \voa{} $V$ of central charge $24$ with $V_1\neq\{0\}$ can be obtained by an orbifold construction from the Leech lattice \voa{} $V_\Lambda$. This suffices to restrict the possible Lie algebras that can occur as weight-one space of $V$ to the $71$ of Schellekens.

Moreover, the fact that each \strathol{} \voa{} $V$ of central charge $24$ comes from the Leech lattice $\Lambda$ can be used to classify these \voa{}s by studying properties of the Leech lattice. We demonstrate this for $43$ of the $70$ non-zero Lie algebras on Schellekens' list, omitting those cases that are too computationally expensive.
\end{abstract}

\maketitle

\setcounter{tocdepth}{1}
\tableofcontents

\section{Introduction}
The weight-one subspace $V_1$ of a \strathol{} \voa{} $V$ of central charge $24$ is a finite-dimensional, reductive Lie algebra \cite{DM04b}. More precisely, $V_1$ is either zero, $24$-dimensional abelian or semisimple $\g_1\oplus\ldots\oplus\g_r$ of rank at most $24$ \cite{DM04}. In $1993$ Schellekens \cite{Sch93} (see also \cite{EMS20a}) showed that there are at most $71$ possibilities for this Lie algebra by exploiting the modular invariance of the character of $V$ (see \autoref{table:69}).

By the works of many authors over three decades the following is now proved:
\begin{thn*}
Each potential Lie algebra on Schellekens' list is realised by a \strathol{} \voa{} of central charge $24$ and this \voa{} is uniquely determined by its $V_1$-structure if $V_1\neq\{0\}$.
\end{thn*}
In addition, uniform existence and uniqueness proofs were recently given in \cite{Hoe17,MS19,HM20}.

Schellekens' work is based on decomposing the \voa{} $V$ into irreducible modules for the affine \voa{} $\langle V_1\rangle\cong L_{\hat\g_1}(k_1,0)\otimes\ldots\otimes L_{\hat\g_r}(k_r,0)$ generated by $V_1$ (with levels $k_i\in\Ns$, see \cite{DM06b}), assuming $V_1$ to be semisimple, and deriving a set of trace identities (cf.\ Theorem~6.1 in \cite{EMS20a}) by viewing the character of $V$ as a Jacobi form. The lowest-order trace identity is
\begin{equation}\label{eq:221i}\tag{1}
\frac{h_i^\vee}{k_i}=\frac{\dim(V_1)-24}{24}
\end{equation}
for all $i=1,\ldots,r$ where $h_i^\vee$ is the dual Coxeter number of $\g_i$ (see also \cite{DM04}). This equation has exactly $221$ solutions, listed in \autoref{table:221}. The higher-order trace identities translate to a system of linear equations on the multiplicities appearing in the decomposition of $V$ into $\langle V_1\rangle$-modules. By excluding those Lie algebras for $V_1$ whose corresponding system has no solutions in the non-negative integers one essentially arrives at Schellekens' list of $71$ possible Lie algebras for $V_1$. This approach relies heavily on solving linear programming problems on the computer.

\medskip

The main goal of this paper is to provide a novel, simpler proof of Schellekens' list by showing that each \strathol{} \voa{} $V$ of central charge $24$ with $V_1\neq\{0\}$ can be obtained as orbifold construction $V\cong V_\Lambda^{\orb(g)}$ associated with the Leech lattice \voa{} $V_\Lambda$ (itself one of these \voa{}s) and some finite-order automorphism $g\in\Aut(V_\Lambda)$.

The orbifold construction \cite{EMS20a,Moe16} is an important method to construct \voa{}s. Let $V$ be a \strathol{} \voa{} and $g$ an automorphism of $V$ of finite order~$n$ and type~$0$. Then the \fpvosa{} $V^g$ is \strat{} \cite{Miy15,CM16} and has exactly $n^2$ non-isomorphic irreducible modules, which can be realised as the eigenspaces of $g$ acting on the irreducible twisted $V$-modules $V(g^j)$, $j\in\Z_n$. If the twisted modules $V(g^j)$ have positive conformal weight for $j\neq0\pmod{n}$, then the sum $V^{\orb(g)}:=\oplus_{j\in\Z_n}V(g^j)^g$ is again a \strathol{} \voa{} (of the same central charge as $V$).

\medskip

Our approach is motivated by the classification of positive-definite, even, unimodular lattices, which bears similarities to the classification of \strathol{} \voa{}s. In $1973$ Niemeier classified the positive-definite, even, unimodular lattices of rank $24$ \cite{Nie73} (see also \cite{Ven80,CS99}). He proved that up to isomorphism there are exactly $24$ such lattices and that the isomorphism class is uniquely determined by the root system. The Leech lattice $\Lambda$ is the unique one amongst these lattices without roots.

Conway, Parker and Sloane showed that there is a natural bijection, mediated by the ``holy construction'', between the $23$ classes of deep holes of the Leech lattice $\Lambda$, i.e.\ points in $\Lambda\otimes_\Z\R$ which have maximal distance to $\Lambda$, and the $23$ Niemeier lattices different from the Leech lattice \cite{CPS82,CS82} (see also \cite{Bor85}). Let $d$ be a deep hole corresponding to the Niemeier lattice $N$ and let $n$ denote the order of $d+\Lambda\in(\Lambda\otimes_\Z\R)/\Lambda$. Then $\Lambda^d:=\{\alpha\in\Lambda\,|\,\langle\alpha,d\rangle\in\Z\}$ is an index-$n$ sublattice of $\Lambda$ and of the lattice $\spn_\Z\{\Lambda^d,d\}$ generated by $d$ and $\Lambda^d$, which is isomorphic to the Niemeier lattice $N$. Note that $n$ equals the (dual) Coxeter number $h$ of (any irreducible component of) the root system of $N$.

For the inverse construction, given the Niemeier lattice $N$, let $u:=\sum_{i=1}^r\rho_i/h$ where the $\rho_i$ denote the Weyl vectors, i.e.\ the sums of the fundamental weights or the half sums of the positive roots, of the irreducible components of the root system of $N$. Then $N^u$ is an index-$h$ sublattice of $N$ that isomorphic to $\Lambda^d$ where $d$ is the corresponding deep hole.

\medskip

In \cite{MS19} the deep-hole construction of the Niemeier lattices was generalised to \strathol{} \voa{}s $V$ of central charge $24$. For each of the $71$ Lie algebras $\g$ on Schellekens' list a \gdh{}, a certain automorphism of the Leech lattice \voa{} $V_\Lambda$, such that $(V_\Lambda^{\orb(g)})_1\cong\g$ is explicitly listed. This provides a uniform proof of the existence part of the above theorem.

Complementarily, in this paper we lift the inverse construction to the level of \voa{}s (generalising \cite{LS20}). First, combining the dimension formula in \cite{MS19} with the ``very strange formula'' in \cite{Kac90} we prove:
\begin{thn*}[\autoref{thm:dimstrange}]
Let $V$ be a \strathol{} \voa{} of central charge $24$ whose weight-one Lie algebra $V_1=\bigoplus_{i=1}^r\g_i$ is semisimple with Cartan subalgebra $\hh$ and $g\neq\id$ an automorphism of $V$ of order~$n$ such that $g|_{V_1}$ is inner and characterised by $\delta=\sum_{i=1}^r\delta_i\in\hh^*$. Assume furthermore that $g$ is of type~$0$, that $V^g$ satisfies the positivity condition and that $g|_{V_1}$ is quasirational. Then
\begin{equation*}
\dim(V^{\orb(g)}_1)\leq24+12n\sum_{i=1}^rh^\vee_i\left|\delta_i-\frac{\rho_i}{h_i^\vee}\right|^2
\end{equation*}
with Weyl vectors $\rho_i$ and dual Coxeter numbers $h_i^\vee$.
\end{thn*}

Then, given a \strathol{} \voa{} $V$ of central charge $24$ with semisimple weight-one Lie algebra $V_1=\bigoplus_{i=1}^r\g_i$, we define the inner automorphism $\sigma_u=\e^{(2\pi\i) u_0}\in\Aut(V)$ with $u=\sum_{i=1}^r\rho_i/h_i^\vee$. Using equation~\eqref{eq:221i} one can show that this automorphism has type~$0$ so that the orbifold construction exists, and the above formula shows that $\dim(V^{\orb(\sigma_u)}_1)\leq24$, which implies that $V^{\orb(\sigma_u)}\cong V_\Lambda$ by \cite{DM06b,DM04b}. The inverse orbifold construction \cite{EMS20a,LS19} implies (not using Schellekens' classification):
\begin{thn*}[\autoref{thm:main1}]
Let $V$ be a \strathol{} \voa{} of central charge $24$ with $V_1\neq\{0\}$. Then $V$ is isomorphic to $V_\Lambda^{\orb(g)}$ for some automorphism $g\in\Aut(V_\Lambda)$.
\end{thn*}
More precisely, the automorphism $g$ may be taken to be a \gdh{} as introduced in \cite{MS19}. As an application we reprove Schellekens' list, only using equation~\eqref{eq:221i} but none of the higher-order trace identities.
\begin{thn*}[\autoref{thm:main2}]
Let $V$ be a \strathol{} \voa{} of central charge $24$ with $V_1\neq\{0\}$. Then $V_1$ is isomorphic to one of the $70$ non-zero Lie algebras in Table~1 of \cite{Sch93}.
\end{thn*}
This novel proof is much simpler than the original proof \cite{Sch93,EMS20a} and for the most part does not rely on computer calculations.

\medskip

Finally, we can also use the result that each \strathol{} \voa{} $V$ of central charge $24$ with $V_1\neq\{0\}$ comes from the Leech lattice $\Lambda$ to prove the uniqueness of such a \voa{} for a given semisimple Lie algebra $V_1$. One merely has to classify the corresponding \gdh{}s in $\Aut(V_\Lambda)$. We demonstrate this for $43$ of the $70$ non-zero Lie algebras on Schellekens' list (see \autoref{thm:main3}). For Lie algebras $V_1$ of larger rank this method encounters computational limitations.

\subsection*{Outline}

In \autoref{sec:24} we recall some results on the classification of \strathol{} \voa{}s of central charge $24$.

In \autoref{sec:dimstrange} we review the cyclic orbifold construction. Then we combine the dimension formula in \cite{MS19} with Kac's ``very strange formula'' to derive the dimension formula in \autoref{thm:dimstrange}.

In \autoref{sec:autleech} we review lattice \voa{}s and describe the automorphisms of the Leech lattice \voa{} $V_\Lambda$ up to conjugation.

In \autoref{sec:leechuniversalvertex} we use \autoref{thm:dimstrange} and equation~\eqref{eq:221i} to prove that any \strathol{} \voa{} $V$ of central charge $24$ with $V_1\neq\{0\}$ can be obtained from the Leech lattice \voa{} $V_\Lambda$ by a cyclic orbifold construction (see \autoref{thm:main1}).

In \autoref{sec:newproof} we use \autoref{thm:main1} and the results in \autoref{sec:autleech} to give a simple proof of Schellekens' list of $71$ Lie algebras that can occur as the weight-one space of a \strathol{} \voa{} of central charge $24$ (see \autoref{thm:main2}).

Finally, in \autoref{sec:uniqueness} we prove the uniqueness of a \strathol{} \voa{} of central charge $24$ with a given weight-one space for $43$ of the $70$ non-zero Lie algebras on Schellekens' list by classifying \gdh{}s in $\Aut(V_\Lambda)$ (see \autoref{thm:main3}).

\subsection*{Acknowledgements}
The authors would like to thank Gerald Höhn and Nils Scheithauer for helpful discussions. The third and fourth authors are grateful for the hospitality of Academia Sinica in Taipei, where parts of this research were conducted. The authors thank the organisers of the 2018 RIMS workshop ``VOAs and Symmetries'' in Kyoto, where also part of this work was done. The first author was supported by CNPq grants 409598/2016-0 and 303806/2017-6, the second by grant AS-IA-107-M02 of Academia Sinica and MOST grant 107-2115-M-001-003-MY3 of Taiwan, the third by an AMS-Simons Travel Grant and the fourth by JSPS KAKENHI Grant Numbers JP17K05154, JP19KK0065 and JP20K03505.


\section{Holomorphic \VOA{}s of Central Charge 24}\label{sec:24}
In this section we review some results on \strathol{} \voa{}s of central charge $24$.

A \voa{} $V$ is called \emph{\strat{}} if it is rational (as defined, e.g., in \cite{DLM97}), $C_2$-cofinite (or lisse), self-contragredient (or self-dual) and of CFT-type. Then $V$ is also simple. Moreover, a rational \voa{} $V$ is said to be \emph{holomorphic} if $V$ itself is the only irreducible $V$-module. The central charge of a \strathol{} \voa{} $V$ is necessarily in $8\N$, a simple consequence of Zhu's modular invariance result \cite{Zhu96}. Examples of \strat{} \voa{}s are those associated with positive-definite, even lattices. If the lattice is unimodular, then the associated \voa{} is holomorphic.

\medskip

Let $V=\bigoplus_{n=0}^\infty V_n$ be a \voa{} of CFT-type. Then the zero modes
\begin{equation*}
[a,b]:=a_0b
\end{equation*}
for $a,b\in V_1$ equip $V_1$ with the structure of a Lie algebra. If $V$ is also self-contragredient, then there exists a non-degenerate, invariant bilinear form $\langle\cdot,\cdot\rangle$ on $V$, which is unique up to a non-zero scalar and symmetric \cite{FHL93,Li94}. We normalise the form such that $\langle\vac,\vac\rangle=-1$, where $\vac$ is the vacuum vector of $V$. Then $a_1b=b_1a=\langle a,b\rangle\vac$ for all for $a,b\in V_1$.

If $g\in\Aut(V)$ is an automorphism of the \voa{} $V$, fixing the vacuum vector $\vac\in V_0$ and the Virasoro vector $\omega\in V_2$ by definition, then the restriction of $g$ to $V_1$ is a Lie algebra automorphism.

\medskip

Let $\g$ be a simple, finite-dimensional Lie algebra with invariant bilinear form $(\cdot,\cdot)$ normalised such that $(\alpha,\alpha)=2$ for any long root $\alpha$. The affine Kac-Moody algebra $\hat\g$ associated with $\g$ is the Lie algebra $\hat\g:=\g\otimes\C[t,t^{-1}]\oplus\C K$ with central element $K$ and Lie bracket
\begin{equation*}
[a\otimes t^m,b\otimes t^n]:=[a,b]\otimes t^{m+n}+m(a,b)\delta_{m+n,0}K
\end{equation*}
for $a,b\in\g$, $m,n\in\Z$.

A representation of $\hat\g$ is said to have level $k\in\C$ if $K$ acts as $k\id$. For a dominant integral weight $\lambda\in P_+$ and $k\in\C$ let $L_{\hat\g}(k,\lambda)$ be the irreducible $\hat\g$-module of level $k$ obtained by inducing the irreducible highest-weight $\g$-module $L_\g(\lambda)$ up in a certain way to a $\hat\g$-module and taking its irreducible quotient (see, e.g., \cite{Kac90}).

For $k\in\Ns$, $L_{\hat\g}(k,0)$ admits the structure of a rational \voa{} whose irreducible modules are given by the $L_{\hat\g}(k,\lambda)$ for $\lambda\in P_+^k$, the subset of the dominant integral weights $P_+$ of level at most $k$ \cite{FZ92}. The conformal weight of the module $L_{\hat\g}(k,\lambda)$ is $\rho(L_{\hat\g}(k,\lambda))=\frac{(\lambda+2\rho,\lambda)}{2(k+h^\vee)}$ with Weyl vector $\rho$ and dual Coxeter number $h^\vee$ (see \cite{Kac90}, Corollary~12.8).

For a self-contragredient \voa{} $V$ of CFT-type the commutator formula implies that the modes satisfy
\begin{equation*}
[a_m,b_n]=(a_0b)_{m+n}+m(a_1b)_{m+n-1}=[a,b]_{m+n}+m\langle a,b\rangle\delta_{m+n,0}\id_V
\end{equation*}
for all $a,b\in V_1$, $m,n\in\Z$. Comparing this with the definition above we see that for a simple Lie subalgebra $\g$ of $V_1$ the map $a\otimes t^n\mapsto a_n$ for $a\in\g$ and $n\in\Z$ defines a representation of $\hat\g$ on $V$ of some level $k_\g\in\C$ with $\langle\cdot,\cdot\rangle|_\g=k_\g(\cdot,\cdot)$.

\medskip

Suppose that $V$ is \strat{}. Then it is shown in \cite{DM04b} that the Lie algebra $V_1$ is reductive, i.e.\ a direct sum of a semisimple and an abelian Lie algebra. Moreover, Theorem~3.1 in \cite{DM06b} states that for a simple Lie subalgebra $\g$ of $V_1$ the restriction of $\langle\cdot,\cdot\rangle$ to $\g$ is non-degenerate, the level $k_\g$ is a positive integer, the \vosa{} of $V$ generated by $\g$ is isomorphic to $L_{\hat\g}(k_\g,0)$ and $V$ is an integrable $\hat\g$-module.

Assume in addition that $V$ is holomorphic and of central charge $24$. Then the Lie algebra $V_1$ is zero, abelian of dimension $24$ or semisimple of rank at most $24$ \cite{DM04}. If the Lie algebra $V_1$ is semisimple, then it decomposes into a direct sum
\begin{equation*}
V_1\cong\g_1\oplus\ldots\oplus\g_r
\end{equation*}
of simple ideals $\g_i$ and the \vosa{} $\langle V_1\rangle$ of $V$ generated by $V_1$ is isomorphic to the tensor product of affine \voa{}s
\begin{equation*}
\langle V_1\rangle\cong L_{\hat\g_1}(k_1,0)\otimes\ldots\otimes L_{\hat\g_r}(k_r,0)
\end{equation*}
with levels $k_i:=k_{\g_i}\in\Ns$ and has the same Virasoro vector as $V$. The decomposition of the \voa{} $\langle V_1\rangle$ is called the \emph{affine structure} of $V$, denoted by
\begin{equation*}
\g_{1,k_1}\ldots\g_{r,k_r}
\end{equation*}
with $k_i$ sometimes omitted if it equals $1$.

Since $\langle V_1\rangle\cong L_{\hat\g_1}(k_1,0)\otimes\ldots\otimes L_{\hat\g_r}(k_r,0)$ is rational, $V$ decomposes into the direct sum of finitely many irreducible $\langle V_1\rangle$-modules
\begin{equation*}
V\cong\bigoplus_\lambda m_\lambda L_{\hat\g_1}(k_1,\lambda_1)\otimes\ldots\otimes L_{\hat\g_r}(k_r,\lambda_r)
\end{equation*}
with $m_\lambda\in\N$ and the sum runs over finitely many $\lambda=(\lambda_1,\ldots,\lambda_r)$ with dominant integral weights $\lambda_i\in P_+^{k_i}(\g_i)$, i.e.\ of level at most $k_i$.

\medskip

Let $h_i^\vee$ denote the dual Coxeter number of $\g_i$. The modular invariance of the character of $V$ implies that the ratio $h_i^\vee/k_i$ is independent of $\g_i$. More precisely,
\begin{equation}\label{eq:221}
\frac{h_i^\vee}{k_i}=\frac{\dim(V_1)-24}{24}
\end{equation}
for all $i=1,\ldots,r$, which follows from the lowest-order trace identity in \cite{Sch93} (see also \cite{DM04}). This equation is the only result from \cite{Sch93} that we shall use in this work in order to reprove Schellekens' list. Note that equation~\eqref{eq:221} also implies that the Lie algebra $V_1$ uniquely determines the affine structure, i.e.\ the levels $k_i$.

Together with the crude inequality $(h_i^\vee)^2<4\dim(\g_i)$ and $k_i\geq 1$ equation~\eqref{eq:221} implies that $\dim(V_1)$ is bounded from above, which proves that there are only finitely many affine structures satisfying equation~\eqref{eq:221}. It is then easy to produce a list of the exactly $221$ solutions of this equation (see \autoref{table:221}).

\medskip

Schellekens then narrowed down this list to $69$ possible affine structures (or Lie algebras $V_1$) by solving large integer linear programming problems on the computer that follow from higher-order trace identities (explicitly written down in \cite{EMS20a}, Theorem~6.1). Together with the zero Lie algebra and the $24$-dimensional abelian Lie algebra this gives Schellekens' list of $71$ Lie algebras (see \autoref{table:69}) that occur as the weight-one space of a \strathol{} \voa{} of central charge $24$ \cite{Sch93}.

In \autoref{sec:newproof} we present an alternative approach to proving this result by using the lowest-order trace identity \eqref{eq:221}, the dimension formula from \cite{MS19} and Kac's ``very strange formula''.


\section{Dimension Formula and ``Very Strange Formula''}\label{sec:dimstrange}
In this section we recall the cyclic orbifold construction. Then we state the dimension formula for central charge $24$ from \cite{MS19} and combine it with Kac's ``very strange formula'' to obtain the first main result of this text.

\subsection{Orbifold Construction}
The cyclic orbifold construction \cite{EMS20a,Moe16} is an important tool that allows to construct new \voa{}s from known ones.

Let $V$ be a \strathol{} \voa{} and $G=\langle g\rangle$ a finite, cyclic group of automorphisms of $V$ of order $n\in\Ns$.

By \cite{DLM00} there is an up to isomorphism unique irreducible $g^i$-twisted $V$-module $V(g^i)$ for each $i\in\Z_n$. The uniqueness of $V(g^i)$ implies that there is a representation $\phi_i\colon G\to\Aut_\C(V(g^i))$ of $G$ on the vector space $V(g^i)$ such that
\begin{equation*}
\phi_i(g)Y_{V(g^i)}(v,x)\phi_i(g)^{-1}=Y_{V(g^i)}(g v,x)
\end{equation*}
for all $i\in\Z_n$, $v\in V$. This representation is unique up to an $n$-th root of unity. Denote the eigenspace of $\phi_i(g)$ in $V(g^i)$ corresponding to the eigenvalue $\e^{(2\pi\i)j/n}$ by $W^{(i,j)}$. On $V(g^0)=V$ we choose $\phi_0(g)=g$.

By \cite{DM97} and recent results in \cite{Miy15,CM16} the \fpvosa{} $V^g=W^{(0,0)}$ is again \strat{}. It has exactly $n^2$ irreducible modules, namely the $W^{(i,j)}$, $i,j\in\Z_n$ \cite{MT04}. One can further show that the conformal weight $\rho(V(g))$ of $V(g)$ is in $(1/n^2)\Z$, and we define the type $t\in\Z_n$ of $g$ by $t=n^2\rho(V(g))\pmod{n}$.

Assume for simplicity that $g$ has type~$0$, i.e.\ that $\rho(V(g))\in(1/n)\Z$. Then it is possible to choose the representations $\phi_i$ such that the conformal weights satisfy
\begin{equation*}
\rho(W^{(i,j)})\in\frac{ij}{n}+\Z
\end{equation*}
and $V^g$ has fusion rules
\begin{equation*}
W^{(i,j)}\boxtimes W^{(l,k)}\cong W^{(i+l,j+k)}
\end{equation*}
for all $i,j,k,l\in\Z_n$ (see \cite{EMS20a}, Section~5), i.e.\ the fusion ring of $V^g$ is the group ring $\C[\Z_n\times\Z_n]$. In particular, all irreducible $V^g$-modules are simple currents.

In essence, the results in \cite{EMS20a} show that for cyclic $G\cong\Z_n$ and \strathol{} $V$ the module category of $V^G$ is the \emph{twisted group double} $\mathcal{D}_\omega(G)$ where the $3$-cocycle $[\omega]\in H^3(G,\C^\times)\cong\Z_n$ is determined by the type $t\in\Z_n$. This proves a special case of a conjecture by Dijkgraaf, Vafa, Verlinde and Verlinde \cite{DVVV89} who stated it for arbitrary finite $G$.

\medskip

In general, a simple \voa{} $V$ is said to satisfy the \emph{positivity condition} if the conformal weight $\rho(W)>0$ for any irreducible $V$-module $W\not\cong V$ and $\rho(V)=0$.

Now, if $V^g$ satisfies the positivity condition (it is shown in \cite{Moe18} that this condition is almost automatically satisfied if $V$ is \strat{}), then the direct sum of $V^g$-modules
\begin{equation*}
V^{\orb(g)}:=\bigoplus_{i\in\Z_n}W^{(i,0)}
\end{equation*}
admits the structure of a \strathol{} \voa{} of the same central charge as $V$ and is called \emph{orbifold construction} associated with $V$ and $g$ \cite{EMS20a}. Note that $\bigoplus_{j\in\Z_n}W^{(0,j)}$ gives back the old \voa{} $V$.

\medskip

We briefly describe the \emph{inverse} (or \emph{reverse}) orbifold construction \cite{EMS20a,LS19}. Suppose that the \strathol{} \voa{} $V^{\orb(g)}$ is obtained by an orbifold construction as described above. Then via $\amgis v:=\e^{(2\pi\i)i/n}v$ for $v\in W^{(i,0)}$ we define an automorphism $\amgis$ of $V^{\orb(g)}$ of order~$n$, and the unique irreducible $\amgis^j$-twisted $V^{\orb(g)}$-module is given by $V^{\orb(g)}(\amgis^j)=\bigoplus_{i\in\Z_n}W^{(i,j)}$, $j\in\Z_n$. Then
\begin{equation*}
(V^{\orb(g)})^{\orb(\amgis)}=\bigoplus_{j\in\Z_n}W^{(0,j)}=V,
\end{equation*}
i.e.\ orbifolding with $\amgis$ is inverse to orbifolding with $g$.

\subsection{Dimension Formula and ``Very Strange Formula''}
A dimension formula for the weight-one space $V^{\orb(g)}_1$ in the case of central charge $24$ was derived in \cite{MS19} by pairing the vector-valued character of $V^g$ with a vector-valued Eisenstein series of weight $2$. It implies the following dimension bound:
\begin{prop}[Dimension Formula, \cite{MS19}, Corollary~5.7]\label{prop:dimform}
Let $V$ be a \strathol{} \voa{} of central charge $24$ and $g\neq\id$ an automorphism of $V$ of finite order~$n$ and type~$0$ such that $V^g$ satisfies the positivity condition. Then
\begin{equation*}
\dim(V^{\orb(g)}_1)\leq24+\sum_{d\mid n}c_n(d)\dim(V_1^{g^d})
\end{equation*}
with the $c_n(d)\in\Q$ determined by $\sum_{d\mid n}(t,d)c_n(d)=n/t$ for all $t\mid n$.
\end{prop}
This generalises earlier results for $n=2,3$ in \cite{Mon94} (see also \cite{LS19}), for $n=2,3,5,7,13$ in \cite{Moe16} and for all $n$ such that the modular curve $\Gamma_0(n)\backslash\H^*$ has genus zero in \cite{EMS20b}. An automorphism $g$ such that $\dim(V^{\orb(g)}_1)$ attains this upper bound is called \emph{extremal}.

\medskip

We note that the upper bound from the dimension formula depends only on the weight-one Lie algebra $V_1$ and the action of the automorphism $g\in\Aut(V)$ on it. In the following we combine this upper bound with Kac's ``very strange formula'' to cast the dimension formula in a more Lie theoretic form.

For simplicity, we only state Kac's ``very strange formula'' in the special case of inner automorphisms but the following arguments work equally well in the general case. Recall from Kac's classification of finite-order automorphisms of simple Lie algebras (see \cite{Kac90}, Chapter~8) that an inner automorphism $g$ of a simple Lie algebra $\g$ of type $X_l$ is characterised by a sequence of non-negative, relatively prime integers $s_0,\ldots,s_l$ associated with the nodes of the untwisted affine Dynkin diagram $X_l^{(1)}$ satisfying $\sum_{i=0}^la_is_i=m$ with Kac labels $a_i$. Equivalently, $g$ is characterised by the vector $\delta$ in the dual Cartan subalgebra $\hh^*$ of $\g$ defined by $(\alpha_i,\delta)=s_i/m$ for all simple roots $\alpha_i$ of $\g$. Identifying the Cartan subalgebra $\hh$ with its dual $\hh^*$ via $(\cdot,\cdot)$ we may view $\delta\in\hh$ and then $g$ is conjugate to $\e^{\ad_{(2\pi\i)\delta}}$.

Recall that the Weyl vector $\rho\in\hh^*$ is the sum of the fundamental weights or the half sum of the positive roots of $\g$ and that $h^\vee$ denotes the dual Coxeter number.
\begin{prop}[``Very Strange Formula'', formula~(12.3.6) in \cite{Kac90}]
Let $\g$ be a finite-dimensional, simple Lie algebra with Cartan subalgebra $\hh$ and $g$ an inner automorphism of $\g$ of order $m\in\Ns$ characterised by $\delta\in\hh^*$ as described above. Then
\begin{equation*}
\frac{\dim(\g)}{24}-\frac{1}{4m^2}\sum_{j=1}^{m-1}j(m-j)\dim(\g_{(j)})=\frac{h^\vee}{2}\left|\delta-\frac{\rho}{h^\vee}\right|^2
\end{equation*}
where $\g_{(j)}$ is the eigenspace $\g_{(j)}=\{x\in\g\,|\,g x=\e^{(2\pi\i)j/m}x\}$, $j=0,\ldots,m-1$.
\end{prop}
Note that the squared norm on the right-hand side is formed with respect to the bilinear form on $\hh^*$ induced by $(\cdot,\cdot)$.

\medskip

An automorphism $g$ of $\g$ of order $m$ is called \emph{quasirational} (see \cite{Kac90}, Section~8.8) if the corresponding eigenspaces satisfy $\dim(\g_{(i)})=\dim(\g_{(j)})$ whenever $(i,m)=(j,m)$ or equivalently if the characteristic polynomial of $g$ has rational coefficients. This means that the characteristic polynomial of $g$ can be written as $\prod_{t\mid m}(x^t-1)^{b_t}$ for some $b_t\in\Z$ (see, e.g., Exercise~13.7 in \cite{Kac90}). In this case we say $g$ has cycle shape $\prod_{t\mid m}t^{b_t}$ and it is not difficult to see that the left-hand side of Kac's ``very strange formula'' may be rewritten as
\begin{equation*}
\frac{\dim(\g)}{24}-\frac{1}{4m^2}\sum_{j=1}^{m-1}j(m-j)\dim(\g_{(j)})=\frac{1}{24}\sum_{t\mid m}\frac{b_t}{t}.
\end{equation*}

\medskip

We are now in a position to prove the first main result of this text:
\begin{thm}\label{thm:dimstrange}
Let $V$ be a \strathol{} \voa{} of central charge $24$ whose weight-one Lie algebra $V_1=\bigoplus_{i=1}^r\g_i$ is semisimple with Cartan subalgebra $\hh$ and $g\neq\id$ an automorphism of $V$ of order~$n$ such that $g|_{V_1}$ is inner and characterised by $\delta=\sum_{i=1}^r\delta_i\in\hh^*$. Assume furthermore that $g$ is of type~$0$, that $V^g$ satisfies the positivity condition and that $g|_{V_1}$ is quasirational. Then
\begin{equation*}
\dim(V^{\orb(g)}_1)\leq24+12n\sum_{i=1}^rh^\vee_i\left|\delta_i-\frac{\rho_i}{h_i^\vee}\right|^2
\end{equation*}
with Weyl vectors $\rho_i$ and dual Coxeter numbers $h_i^\vee$.
\end{thm}
\begin{proof}
Since $g|_{V_1}$ is quasirational (and of some order $m\mid n$), Kac's ``very strange formula'', generalised to the semisimple Lie algebra $V_1$, becomes
\begin{equation*}
\frac{1}{24}\sum_{t\mid m}\frac{b_t}{t}=\sum_{i=1}^r\frac{h_i^\vee}{2}\left|\delta_i-\frac{\rho_i}{h_i^\vee}\right|^2
\end{equation*}
where $\prod_{t\mid m}t^{b_t}$ is the cycle shape of $g|_{V_1}$. Also note that
\begin{equation*}
\dim(V_1^{g^d})=\sum_{t\mid m}b_t(t,d)
\end{equation*}
for any $d\in\Z$. The dimension formula then yields
\begin{align*}
\dim(V^{\orb(g)}_1)&\leq24+\sum_{d\mid n}c_n(d)\dim(V_1^{g^d})=24+\sum_{d\mid n}c_n(d)\sum_{t\mid m}b_t(t,d)\\
&=24+\sum_{t\mid m}b_t\sum_{d\mid n}c_n(d)(t,d)=24+n\sum_{t\mid m}\frac{b_t}{t}\\
&=24+12n\sum_{i=1}^rh^\vee_i\left|\delta_i-\frac{\rho_i}{h_i^\vee}\right|^2
\end{align*}
where we used the defining relations of the $c_n(d)$.
\end{proof}


\section{Lattice \VOA{}s}\label{sec:autleech}
In this section we describe lattice \voa{}s, the automorphisms of the Leech lattice \voa{} $V_\Lambda$ \cite{Bor86,FLM88} and in particular their conjugacy classes, which were classified in \cite{MS19}.

For any positive-definite, even lattice $L$ (with bilinear form $\langle\cdot,\cdot\rangle\colon L\times L\to\Z$) the associated \voa{} is given by
\begin{equation*}
V_L=M(1)\otimes\C_\eps[L]
\end{equation*}
with the Heisenberg \voa{} $M(1)$ of rank $\rk(L)$ associated with $\h=L\otimes_\Z\C$ and the twisted group algebra $\C_\eps[L]$, the algebra with basis $\{\ee_\alpha\,|\,\alpha\in L\}$ and product $\ee_\alpha\ee_\beta=\eps(\alpha,\beta)\ee_{\alpha+\beta}$ for all $\alpha,\beta\in L$ where $\eps\colon L\times L\to\{\pm1\}$ is a choice of 2-cocycle satisfying $\eps(\alpha,\beta)/\eps(\beta,\alpha)=(-1)^{\langle\alpha,\beta\rangle}$.

Let $\O(L)$ denote the orthogonal group (or automorphism group) of the lattice $L$. For $\nu\in\O(L)$ and a function $\eta\colon L\to\{\pm1\}$ the map $\phi_\eta(\nu)$ acting on $\C_\eps[L]$ as $\phi_\eta(\nu)(\ee_\alpha)=\eta(\alpha)\ee_{\nu\alpha}$ for $\alpha\in L$ and as $\nu$ on $M(1)$ defines an automorphism of $V_L$ if and only if
\begin{equation*}
\frac{\eta(\alpha)\eta(\beta)}{\eta(\alpha+\beta)}=\frac{\eps(\alpha,\beta)}{\eps(\nu\alpha,\nu\beta)}
\end{equation*}
for all $\alpha,\beta\in L$. In this case $\phi_\eta(\nu)$ is called a \emph{lift} of $\nu$ and all such automorphisms form the subgroup $\O(\hat{L})$ of $\Aut(V_L)$. There is a short exact sequence
\begin{equation*}
1\to\Hom(L,\{\pm1\})\to\O(\hat{L})\to\O(L)\to1
\end{equation*}
with the surjection $\O(\hat{L})\to\O(L)$ given by $\phi_\eta(\nu)\mapsto\nu$. The image of $\Hom(L,\{\pm1\})$ in $\O(\hat{L})$ are exactly the lifts of $\id\in\O(L)$. For later use let $\nu\mapsto\hat{\nu}$ denote a fixed section $\O(L)\to\O(\hat{L})$.

If the restriction of $\eta$ to the fixed-point lattice $L^\nu$ is trivial, we call $\phi_\eta(\nu)$ a \emph{standard lift} of $\nu$. It is always possible to choose $\eta$ in this way (see \cite{Lep85}, Section~5). It was proved in \cite{EMS20a}, Proposition~7.1, that all standard lifts of a given $\nu\in\O(L)$ are conjugate in $\Aut(V_L)$. For convenience, let us assume that the section $\O(L)\to\O(\hat{L})$, $\nu\mapsto\hat{\nu}$ only maps to standard lifts. This is not essential but simplifies the presentation.

For any \voa{} $V$ of CFT-type $K:=\langle\{\e^{v_0}\,|\,v\in V_1\}\rangle$ defines a subgroup of $\Aut(V)$, called the \emph{inner automorphism group} of $V$. By \cite{DN99}, Theorem~2.1, the automorphism group of $V_L$ is of the form
\begin{equation*}
\Aut(V_L)=\O(\hat{L})\cdot K,
\end{equation*}
$K$ is a normal subgroup of $\Aut(V_L)$, $\Hom(L,\{\pm1\})$ a subgroup of $K\cap\O(\hat{L})$ and $\Aut(V_L)/K$ is isomorphic to some quotient group of $\O(L)$.

\medskip

In the following we specialise to the Leech lattice $\Lambda$, the up to isomorphism unique unimodular, positive-definite, even lattice of rank $24$ without roots, i.e.\ vectors of norm $2$. The automorphism group $\O(\Lambda)$ is the Conway group $\Co_0$. Since $(V_\Lambda)_1=\{h(-1)\otimes\ee_0\,|\,h\in\Lambda_\C\}\cong\Lambda_\C$, the inner automorphism group satisfies
\begin{equation*}
K=\{\e^{h_0}\,|\,h\in\Lambda_\C\}
\end{equation*}
and is abelian. Moreover, $\Aut(V_\Lambda)/K\cong\O(\Lambda)$, i.e.\ there is a short exact sequence
\begin{equation*}
1\to K\to\Aut(V_\Lambda)\to\O(\Lambda)\to1.
\end{equation*}
This is because $K\cap\O(\hat{\Lambda})=\Hom(\Lambda,\{\pm1\})$ in the special case of the Leech lattice. By the above, every automorphism of $V_\Lambda$ is of the form
\begin{equation*}
\phi_\eta(\nu)\sigma_h
\end{equation*}
for a lift $\phi_\eta(\nu)$ of some $\nu\in\O(\Lambda)$ and with $\sigma_h:=\e^{-(2\pi\i)h_0}$ for some $h\in\Lambda_\C$. The surjection $\Aut(V_\Lambda)\to\O(\Lambda)$ in the short exact sequence is given by $\phi_\eta(\nu)\sigma_h\mapsto\nu$.

It suffices to take $\phi_\eta(\nu)=\hat{\nu}$ from the section $\O(\Lambda)\to\O(\hat{\Lambda})$ since any two lifts only differ by a homomorphism $\Lambda\to\{\pm 1\}$, which can be absorbed into $\sigma_h$. Moreover, since $\sigma_h=\id$ if and only if $h\in\Lambda'=\Lambda$, it is enough to take $h\in\Lambda_\C/\Lambda$.

\medskip

We now describe the conjugacy classes of $\Aut(V_\Lambda)$. For $\nu\in\O(\Lambda)$ let $\pi_\nu=\frac{1}{|\nu|}\sum_{i=0}^{|\nu|-1}\nu^i$ denote the projection of $\Lambda_\C$ onto the elements of $\Lambda_\C$ fixed by $\nu$. By Lemma~8.3 in \cite{EMS20b} (see also Lemma~3.4 in \cite{LS19}) the automorphism $\phi_\eta(\nu)\sigma_h$ is conjugate to $\phi_\eta(\nu)\sigma_{\pi_\nu(h)}$ for any $h\in\Lambda_\C$, and $\phi_\eta(\nu)$ and $\sigma_{\pi_\nu(h)}$ commute.

In \cite{MS19} all automorphisms in $\Aut(V_\Lambda)$ were classified up to conjugation. A similar result for arbitrary lattice \voa{}s was proved in \cite{HM20}.
\begin{prop}[\cite{MS19}, Proposition~3.2]\label{prop:autleech}
Let $N$ denote a set of representatives for the conjugacy classes of $\O(\Lambda)\cong\Co_0$. For $\nu\in N$ let $H_\nu$ denote a set of representatives for the orbits of the action of the centraliser $C_{\O(\Lambda)}(\nu)$ on $\pi_\nu(\Lambda_\C)/\pi_\nu(\Lambda)$. Then $(\nu,h)\mapsto\hat{\nu}\sigma_h$ (with a fixed section $\O(\Lambda)\to\O(\hat{\Lambda})$, $\nu\mapsto\hat{\nu}$) defines a bijection from the set $Q:=\{(\nu,h)\,|\,\nu\in N,h\in H_\nu\}$ to the conjugacy classes of $\Aut(V_\Lambda)$.
\end{prop}

We also describe the conjugacy classes in $\Aut(V_\Lambda)$ of a given finite order~$n$. First note that a standard lift $\phi_\eta(\nu)$ of $\nu$ has order $m:=|\nu|$ if $m$ is odd or if $m$ is even and $\langle\alpha,\nu^{m/2}\alpha\rangle\in 2\Z$ for all $\alpha\in\Lambda$ and order $2m$ otherwise. In the latter case we say that $\nu$ exhibits \emph{order doubling}. Then $\phi_\eta(\nu)\ee_\alpha=(-1)^{m\langle\pi_\nu(\alpha),\pi_\nu(\alpha)\rangle}\ee_\alpha=(-1)^{\langle\alpha,\nu^{m/2}\alpha\rangle}\ee_\alpha$ for all $\alpha\in\Lambda$. (Note that $\alpha\mapsto m\langle\pi_\nu(\alpha),\pi_\nu(\alpha)\rangle+2\Z=\langle\alpha,\nu^{m/2}\alpha\rangle+2\Z$ defines a homomorphism $\Lambda\to\Z_2$.)

Given a standard lift $\phi_\eta(\nu)$ that exhibits order doubling there exists a vector $s_\nu\in(1/2m)\Lambda^\nu$ defining an inner automorphism $\sigma_{s_\nu}=\e^{-(2\pi\i)(s_\nu)_0}$ of order $2m$ such that $\phi_\eta(\nu)\sigma_{s_\nu}$ has order $m$. (If $\nu$ does not exhibit order doubling, then we set $s_\nu=0$.) Then the order of an automorphism $\phi_\eta(\nu)\sigma_{s_{\nu}+f}$ for $f\in\Lambda\otimes_\Z\Q$ is given by $\lcm(m,k)$ where $k$ is minimal in $\Ns$ such that $kf$ is in $\Lambda$ or equivalently in the fixed-point lattice $\Lambda^\nu$.

For convenience, we define the $s_\nu$-shifted action of $C_{\O(\Lambda)}(\nu)$ on $\pi_\nu(\Lambda_\C)$ by
\begin{equation*}
\tau.f=\tau f+(\tau-\id)s_\nu
\end{equation*}
for all $\tau\in C_{\O(\Lambda)}(\nu)$ and $f\in\pi_\nu(\Lambda_\C)$. The following result is immediate:
\begin{prop}\label{prop:autleech2}
A complete system of representatives for the conjugacy classes of automorphisms in $\Aut(V_\Lambda)$ of order $n\in\Ns$ consists of the $\hat{\nu}\sigma_{s_{\nu}+f}$ where
\begin{enumerate}
\item $\nu$ is from the representatives in $N\subset\O(\Lambda)$ of order $m$ dividing $n$,
\item $f$ is from the orbit representatives of the $s_\nu$-shifted action of $C_{\O(\Lambda)}(\nu)$ on $(\Lambda^\nu/n)/\pi_{\nu}(\Lambda)$
\end{enumerate}
such that $\lcm(m,|\sigma_f|)=n$.
\end{prop}

\medskip

We conclude this section by recalling some results on the twisted modules of lattice \voa{}s. For a standard lift $\phi_\eta(\nu)$ the irreducible $\phi_\eta(\nu)$-twisted modules of a lattice \voa{} $V_L$ are described in \cite{DL96,BK04}. Together with the results in Section~5 of \cite{Li96} this allows us to describe the irreducible $g$-twisted $V_L$-modules for all finite-order automorphisms $g\in\Aut(V_L)$.

For simplicity, let $L$ be unimodular. Then $V_L$ is holomorphic and there is a unique irreducible $g$-twisted $V_L$-module $V_L(g)$ for each $g\in\Aut(V_L)$ of finite order. Let $g=\phi_\eta(\nu)\sigma_h$ for some standard lift $\phi_\eta(\nu)$ and $\sigma_h=\e^{-(2\pi\i)h_0}$ for some $h\in\pi_\nu(L\otimes_\Z\Q)$. Then
\begin{equation*}
V_{L}(g)=M(1)[\nu]\otimes\ee_{h}\C[\pi_\nu(L)]\otimes\C^{d(\nu)}
\end{equation*}
with twisted Heisenberg module $M(1)[\nu]$, grading by the lattice coset $h+\pi_\nu(L)$ and defect $d(\nu)\in\Ns$.

Assume that $\nu$ has order~$m$ and cycle shape $\prod_{t\mid m}t^{b_t}$ with $b_t\in\Z$, i.e.\ the extension of $\nu$ to $\h$ has characteristic polynomial $\prod_{t\mid m}(x^t-1)^{b_t}$. Then the conformal weight of $V_L(g)$ is given by
\begin{equation*}
\rho(V_L(g))=\frac{1}{24}\sum_{t\mid m}b_t\left(t-\frac{1}{t}\right)+\min_{\alpha\in h+\pi_{\nu}(L)}\langle\alpha,\alpha\rangle/2\geq0,
\end{equation*}
where $\rho_\nu:=\frac{1}{24}\sum_{t\mid m}b_t\left(t-\frac{1}{t}\right)$ is called the \emph{vacuum anomaly} of $V_L(g)$ \cite{DL96}. Note that $\rho_\nu$ is positive for $\nu\neq\id$.


\section{Neighbours of the Leech Lattice \VOA{}}\label{sec:leechuniversalvertex}
In this section, given a \strathol{} \voa{} $V$ of central charge $24$ with $V_1=\bigoplus_{i=1}^r\g_i$ semisimple, we define a certain inner automorphism $\sigma_u\in\Aut(V)$ and show using \autoref{thm:dimstrange} and equation~\eqref{eq:221} that the corresponding orbifold construction $V^{\orb(\sigma_u)}$ is isomorphic to the Leech lattice \voa{} $V_\Lambda$. By the inverse orbifold construction this then proves that every \strathol{} \voa{} $V$ of central charge $24$ with $V_1\neq\{0\}$ can be obtained by an orbifold construction from the Leech lattice \voa{} $V_\Lambda$.

The upper bound in \autoref{thm:dimstrange} suggests a canonical choice for the automorphism $\sigma_u\in\Aut(V)$, namely
\begin{equation*}
\sigma_u=\e^{(2\pi\i)u_0}\quad\text{with}\quad u:=\sum_{i=1}^r\rho_i/h_i^\vee
\end{equation*}
with Weyl vectors $\rho_i$ and dual Coxeter numbers $h_i^\vee$, where we view $u$ in the Cartan subalgebra $\hh$ of $V_1$ by means the identification of $\hh$ with $\hh^*$ via the invariant bilinear form $(\cdot,\cdot)$, which we normalised on each simple ideal $\g_i$ such that long roots have norm $2$.

First, we study some properties of this automorphism. We denote by $l_i\in\{1,2,3\}$ the lacing number of $\g_i$.
\begin{prop}\label{prop:prepremain1_1}
Let $V$ be a \strathol{} \voa{} of central charge $24$ with $V_1=\bigoplus_{i=1}^r\g_i$ semisimple. Then the inner automorphism $\sigma_u=\e^{(2\pi\i)u_0}$ with $u=\sum_{i=1}^r\rho_i/h_i^\vee$ has order $n=\lcm(\{l_ih_i^\vee\}_{i=1}^r)$ (as does the restriction $\sigma_u|_{V_1}$ to $V_1$) and type~$0$, and $\sigma_u|_{V_1}$ is quasirational.
\end{prop}
\begin{proof}
First, we describe the Lie algebra automorphism $\sigma_u|_{V_1}=\e^{\ad_{(2\pi\i)u}}$ in the language of \cite{Kac90}, Chapter~8. Since the Weyl vector $\rho_i$ is the sum of the fundamental weights (dual to the simple coroots) of $\g_i$, it is easy to determine the sequence of integers $s_0,\ldots,s_l$ (or equivalently the vector $\delta_i$ in the dual Cartan subalgebra $\hh_i^*$ of $\g_i$) characterising the action of $\sigma_u|_{\g_i}=\e^{\ad_{(2\pi\i)\rho_i/h_i^\vee}}$ on the simple roots of $\g_i$. Indeed, if $\g_i$ is simply-laced, i.e.\ of type $A_l$, $D_l$ or $E_l$, then $\sigma_u|_{\g_i}$ has type $(1,\ldots,1;1)$ (and this is the up to conjugation unique minimal-order regular automorphism of $\g_i$, see Exercise~8.11 in \cite{Kac90}). For types $B_l$, $C_l$, $F_4$ or $G_2$ the automorphism $\sigma_u|_{\g_i}$ has type $(2,\ldots,2,1;1)$, $(2,1,\ldots,1,2;1)$, $(2,2,2,1,1;1)$ or $(3,3,1;1)$, respectively. Not surprisingly, this shows that $\delta_i=\rho_i/h_i^\vee$.

The quasirationality of $\sigma_u|_{V_1}$ now follows from the quasirationality on each simple ideal $\g_i$. If $\g_i$ is simply-laced, this is Exercise~8.12 in \cite{Kac90}. For the non simply-laced cases, it is not difficult to show that an automorphism of the above-mentioned types is quasirational.

\medskip

The order of $\sigma_u|_{\g_i}$ is $l_ih_i^\vee$ with dual Coxeter number $h_i^\vee$ and lacing number $l_i$ so that the Lie algebra automorphism $\sigma_u|_{V_1}$ has order $n=\lcm(\{l_ih_i^\vee\}_{i=1}^r)$.

In principle, the order of $\sigma_u$ on the \voa{} $V$ is a multiple of $n$. In the following we show that in fact $|\sigma_u|=n=\lcm(\{l_ih_i^\vee\}_{i=1}^r)$. Recall that $V$ decomposes into a direct sum of irreducible $\langle V_1\rangle$-modules $V\cong\bigoplus_\lambda m_\lambda \bigotimes_{i=1}^rL_{\hat\g_1}(k_i,\lambda_i)$ where the sum runs over weights $\lambda=(\lambda_1,\ldots,\lambda_r)$ with dominant integral weights $\lambda_i\in P_+^{k_i}(\g_i)$. In order to prove that $\sigma_u$ has order $n$, it suffices to show that $n(u,\lambda)\in\Z$ for all weights $\lambda$ appearing in this decomposition. (It is easy to see that $2n(u,\lambda)\in\Z$ for any weight $\lambda$ in the weight lattice $P$.)

The conformal weight of a module $\bigotimes_{i=1}^rL_{\hat\g_i}(k_i,\lambda_i)$ (see \autoref{sec:24}) appearing in the decomposition of $V$ must be an integer, i.e.
\begin{align*}
\sum_{i=1}^r\frac{(2\rho_i+\lambda_i,\lambda_i)}{2(k_i+h_i^\vee)}&=\frac{h_i^\vee}{2(h_i^\vee+k_i)}\sum_{i=1}^r\left(2(\rho_i/h_i^\vee,\lambda_i)+\frac{1}{h_i^\vee}
(\lambda_i,\lambda_i)\right)\\
&=\frac{1}{2n(1+k_i/h_i^\vee)}\left(2n(u,\lambda)+\sum_{i=1}^r\frac{n}{h_i^\vee}(\lambda_i,\lambda_i)\right)\in\Z
\end{align*}
where we used that $h_i^\vee/(h_i^\vee+k_i)$ is independent of $i$ by equation~\eqref{eq:221}. Then, since $n(1+k_i/h_i^\vee)\in\Z$, this implies that $2n(u,\lambda)+\sum_{i=1}^r\frac{n}{h_i^\vee}(\lambda_i,\lambda_i)\in\Z$ and, recalling that $2n(u,\lambda)\in\Z$, we obtain
\begin{equation*}
\sum_{i=1}^r\frac{n}{h_i^\vee}(\lambda_i,\lambda_i)\in\Z.
\end{equation*}
Because the Weyl vector $\rho_i=\frac{1}{2}\sum_{\alpha\in\Phi_i^+}\alpha$ is the half sum of positive roots,
\begin{align*}
(2n(u,\lambda))^2&=\left(\sum_{i=1}^r\frac{n}{h_i^\vee}(2\rho_i,\lambda_i)\right)^2=\left(\sum_{i=1}^r\sum_{\alpha\in\Phi_i^+}\frac{n}{h_i^\vee}(\alpha,\lambda_i)\right)^2\\
&\equiv\sum_{i=1}^r\sum_{\alpha\in\Phi_i^+}\left(\frac{n}{h_i^\vee}\right)^2(\alpha,\lambda_i)^2\pmod{2}\\
&=\sum_{i=1}^r\left(\frac{n}{h_i^\vee}\right)^2h_i^\vee(\lambda_i,\lambda_i)=n\sum_{i=1}^r\frac{n}{h_i^\vee}(\lambda_i,\lambda_i).
\end{align*}
Note that $\frac{n}{h_i^\vee}(\alpha,\lambda_i)\in\Z$ for any root $\alpha\in\Phi_i$.

Here, we used that for any irreducible Lie algebra $\g$ (with Cartan subalgebra $\hh$, root system $\Phi\subseteq\hh^*$, positive roots $\Phi^+\subseteq\Phi$ and dual Coxeter number $h^\vee$) the identity
\begin{equation*}
\sum_{\alpha\in\Phi^+}(\alpha,y)^2=\frac{1}{2}\sum_{\alpha\in\Phi}(\alpha,y)^2=h^\vee(y,y)
\end{equation*}
holds for all $y\in\hh^*$. This follows from Exercise~6.2 in \cite{Kac90}, which states that the Killing form $\tr(\ad_y\ad_y)$ is equal to $2h^\vee(y,y)$.

Now, if $n$ is even, then $n\sum_{i=1}^r\frac{n}{h_i^\vee}(\lambda_i,\lambda_i)$ is also even and so is $2n(u,\lambda)^2$. Thus, $2n(u,\lambda)$ is even and $n(u,\lambda)$ is an integer. If $n$ is odd, then all lacing numbers $l_i$ and dual Coxeter numbers $h_i^\vee$ are odd. This means that all irreducible components of $V_1$ are of type $A_l$ with $l$ even. However, $\rho_i$ is in the root lattice in this case. Therefore, $n(u,\lambda)\in\Z$ for any $\lambda$ in the weight lattice $P$.

\medskip

To prove that $\sigma_u$ has type~$0$ we use the ``strange formula'' of Freudenthal--de~Vries $|\rho_i|^2/(2h_i^\vee)=\dim(\g_i)/24$ (see, e.g., equation~(12.1.8) in \cite{Kac90}), which is a special case of the ``very strange formula'', and equation~\eqref{eq:221}. The $L_0$-operator of the twisted module $V(\sigma_u)$ is shifted in comparison to $V$ by $-u_0+(1/2)\langle u,u\rangle\id$ (see \cite{Li96}, Section~5). Since $\sigma_u$ has order~$n$, $u_0$ has eigenvalues in $(1/n)\Z$, and hence $\sigma_u$ is of type~$0$ if and only if $\langle u,u\rangle/2\in(1/n)\Z$. We compute
\begin{equation*}
\frac{1}{2}\langle u,u\rangle=\sum_{i=1}^r\frac{\langle\rho_i,\rho_i\rangle}{2(h_i^\vee)^2}=\sum_{i=1}^r\frac{k_i(\rho_i,\rho_i)}{2(h_i^\vee)^2}=\sum_{i=1}^r\frac{k_i\dim(\g_i)}{24h_i^\vee}
\end{equation*}
where we used the ``strange formula'' in the last step. With equation~\eqref{eq:221} this becomes
\begin{equation*}
\frac{1}{2}\langle u,u\rangle=\sum_{i=1}^r\frac{\dim(\g_i)}{\dim(V_1)-24}=\frac{\dim(V_1)}{\dim(V_1)-24}=\frac{h_i^\vee+k_i}{h_i^\vee}=\frac{l_i(h_i^\vee+k_i)}{l_ih_i^\vee}
\end{equation*}
for all $i=1,\ldots,r$. Since $k_i\in\Z$ by \cite{DM06b} and $h_i^\vee\in\Z$, this implies that $\langle u,u\rangle/2\in(1/n)\Z$. Hence $\sigma_u$ is of type~$0$.
\end{proof}

The following result is probably known:
\begin{lem}\label{lem:pos}
Let $V$ be a simple \voa{} and $U$ a full \vosa{} of $V$ that is simple and rational. If $U$ satisfies the positivity condition, then so does $V$.
\end{lem}
\begin{proof}
By the rationality of $U$, every $V$-module decomposes into a direct sum of irreducible $U$-modules. Because the conformal weights $\rho(M)$ are non-negative for all irreducible $U$-modules $M$, the same is true for all irreducible $V$-modules.

Now, let $M$ be an irreducible $V$-module with $\rho(M)=0$. Then, as $U$-module, $M$ contains $U$ since $U$ is rational and $U$ is the only irreducible $U$-module with non-positive conformal weight. The vacuum vector $v:=\vac\in U$ is a vacuum-like vector of $U$ when we view $U$ as a $U$-module, i.e. $L_{-1}u=0$. Since $U$ and $V$ have the same Virasoro vector by assumption, $v$ is also a vacuum-like vector of the $V$-module $M$. Then, by Proposition~3.4 in \cite{Li94} there is a non-zero $V$-module homomorphism from $V$ to $M$ and since both $V$ and $M$ are irreducible, $V\cong M$ as $V$-modules by Schur's lemma. Hence, $V$ satisfies the positivity condition.
\end{proof}
\begin{prop}\label{prop:prepremain1_2}
Let $V$ be a \strathol{} \voa{} of central charge $24$ with $V_1=\bigoplus_{i=1}^r\g_i$ semisimple and consider the inner automorphism $\sigma_u=\e^{(2\pi\i)u_0}$ with $u=\sum_{i=1}^r\rho_i/h_i^\vee$. Then the \fpvosa{} $V^{\sigma_u}$ satisfies the positivity condition.
\end{prop}
\begin{proof}
$V$ contains the full \vosa{} $\langle V_1\rangle\cong\bigotimes_{i=1}^rL_{\hat\g_i}(k_i,0)$, which in turn contains the full \vosa{} $V_Q\otimes W$ where $V_Q$ is the lattice \voa{} associated with the lattice $Q:=\bigoplus_{i=1}^r\sqrt{k_i}Q^l_i$, $Q^l_i$ is the lattice spanned by the long roots of $\g_i$ normalised to have squared norm $2$ (see, e.g., Corollary~5.8 in \cite{DM06b}) and $W=\Com_{\langle V_1\rangle}(V_Q)$ is the commutant (or centraliser) of $V_Q$ in $\langle V_1\rangle$. By definition, $W\cong\bigotimes_{i=1}^r K(\g_i,k_i)$ where $K(\g_i,k_i)$ denotes the parafermion \voa{} associated with $\g_i$ at level $k_i$.

Since $\sigma_u=\e^{(2\pi\i)u_0}$ is an inner automorphism associated with an element $u$ in the Cartan subalgebra of $V_1$, the \fpvosa{} $V^{\sigma_u}$ contains the full \vosa{} $V_{Q^u}\otimes W$ with lattice $Q^u=\{\alpha\in Q\,|\,\langle\alpha,u\rangle\in\Z\}$. It is well-known that lattice \voa{}s are rational and satisfy the positivity condition. For parafermion \voa{}s this is shown in \cite{DR17} (using the rationality results in \cite{CM16}). Therefore, $V_{Q^u}\otimes W$ satisfies the positivity condition, and by \autoref{lem:pos} so does $V^{\sigma_u}$.
\end{proof}
\begin{rem} We sketch another proof of \autoref{prop:prepremain1_2}, based on the geometry of affine Weyl groups. Let $\g$ be a simple Lie algebra and $\lambda$ a dominant integral weight in $P_+^k$, $k\in\Ns$. We identify the Cartan subalgebra $\hh$ with its dual $\hh^*$ via the form $(\cdot,\cdot)$. For any non-zero vector in the integrable $\hat\g$-module $M=L_{\hat\g}(k,\lambda)$, homogeneous of weight $\mu\in\hh^*$ and of $L_0$-eigenvalue $\Delta$, the bound
\begin{equation*}
\Delta\geq\frac{|\mu|^2}{2k}+\left(\frac{(\lambda, \lambda+2\rho)}{2(k+h^\vee)}-\frac{|\lambda|^2}{2k}\right)
\end{equation*}
must be satisfied, with equality only possible for $\mu\in\lambda+kQ^\vee$ where $Q^\vee$ denotes the coroot lattice. The weights of the twisted module $M(\sigma_u)$, $u=\rho/h^\vee$, satisfy a similar inequality with $\mu+ku$ in place of $\mu$. The term in parentheses is non-negative for all $\lambda\in P_+^k$, vanishing only if $\lambda=k\Lambda_i$ where $\Lambda_i$ is a fundamental weight with Kac label $a_i=1$. (The Kac labels are the coefficients of the highest root $\theta$ in terms of the simple roots $\alpha_i$ of $\g$.) It is well known that, modulo $Q^\vee$, this set of fundamental weights forms a complete set of representatives of $Q^*/Q^\vee$ where $Q^*$ is the dual of the root lattice \cite{FKW92}.

These remarks, together with the easy fact that the smallest positive integer $p$ for which $p\rho/h^\vee\in Q^*$ is $p=lh^\vee$, imply that $V(\sigma_{pu})$ has positive conformal weight whenever $0<p<lh^\vee=|\sigma_u|$. The same statement for semisimple $\g$ follows, and thus too \autoref{prop:prepremain1_2}.
\end{rem}

Having studied the properties of the automorphism $\sigma_u$ we are now in a position to prove:
\begin{prop}\label{prop:premain1}
Let $V$ be a \strathol{} \voa{} of central charge $24$ with $V_1=\bigoplus_{i=1}^r\g_i$ semisimple and consider the inner automorphism $\sigma_u=\e^{(2\pi\i)u_0}$ with $u=\sum_{i=1}^r\rho_i/h_i^\vee$. Then the corresponding orbifold construction $V^{\orb(\sigma_u)}$ is isomorphic to the Leech lattice \voa{} $V_\Lambda$.
\end{prop}
Note that $\sigma_u$ is exactly the automorphism chosen in Section~2.3 of \cite{LS20} if $V_1$ is simply-laced and reduces to the ``holy construction'' in \cite{CS82} if $V$ is one of the $23$ \voa{}s associated with the Niemeier lattices other than the Leech lattice $\Lambda$.
\begin{proof}
We checked in \autoref{prop:prepremain1_1} and \autoref{prop:prepremain1_2} that $\sigma_u\in\Aut(V)$ admits an orbifold construction and satisfies the assumptions of \autoref{thm:dimstrange}. It follows from the definition of $\sigma_u$ restricted to $V_1$ (recall that we showed $\delta_i=\rho_i/h_i^\vee$) that
\begin{equation*}
\dim(V^{\orb(\sigma_u)}_1)\leq24.
\end{equation*}
On the other hand, $\dim(V^{\orb(\sigma_u)}_1)\geq \dim(V^{\sigma_u}_1)>0$. Then equation~\eqref{eq:221} implies that $V^{\orb(\sigma_u)}_1$ cannot be semisimple so that $V^{\orb(\sigma_u)}_1$ must be $24$-dimensional abelian \cite{DM04}. By Theorem~3 in \cite{DM04b}, $V^{\orb(\sigma_u)}\cong V_\Lambda$.
\end{proof}

\medskip

With the inverse orbifold construction, \autoref{prop:premain1} immediately implies that every \strathol{} \voa{} $V$ of central charge $24$ with $V_1$ semisimple can be obtained by an orbifold construction $V\cong V_\Lambda^{\orb(g)}$ from the Leech lattice \voa{} $V_\Lambda$. More precisely, the properties of the inner automorphism $\sigma_u$ imply that the inverse-orbifold automorphism $g\in\Aut(V_\Lambda)$ is a \emph{\gdh}. \Gdh{}s were introduced in \cite{MS19} and arise naturally as generalisations of the deep holes of the Leech lattice:
\begin{defi}[\GDH{}, \cite{MS19}]\label{def:gendeephole}
Let $V$ be a \strathol{} \voa{} of central charge $24$ and $g\in\Aut(V)$ of finite order $n>1$. Suppose $g$ has type~$0$ and $V^g$ satisfies the positivity condition. Then $g$ is called a \emph{\gdh} of $V$ if
\begin{enumerate}
\item $g$ is extremal, i.e.\ $\dim((V^{\orb(g)})_1)$ attains the upper dimension bound in \autoref{prop:dimform},
\item $\rk(V_1^g)=\rk(V_1^{\orb(g)})$.\label{item:gendeephole3}
\end{enumerate}
\end{defi}
Note that both $V_1^g$ and $V_1^{\orb(g)}$ are reductive by \cite{DM04b}. By convention we also call the identity a \gdh{} so that the Leech lattice \voa{} $V_\Lambda=V_\Lambda^{\orb(\id)}$ may be included in the following theorem.
\begin{thm}\label{thm:main1}
Let $V$ be a \strathol{} \voa{} of central charge $24$ with $V_1\neq\{0\}$. Then $V$ is isomorphic to $V_\Lambda^{\orb(g)}$ for some \gdh{} $g$ in $\Aut(V_\Lambda)$.
\end{thm}
\begin{proof}
We may assume that $V_1$ is semisimple. We showed in \autoref{prop:premain1} that $V^{\orb(\sigma_u)}$ is isomorphic to the Leech lattice \voa{} $V_\Lambda$. Since the upper bound of $24$ in the dimension formula is attained, $\sigma_u$ is extremal. We then consider the inverse-orbifold automorphism $g\in\Aut(V_\Lambda)$ with $V_\Lambda^{\orb(g)}\cong V$. In particular, $g$ is of type~$0$ and satisfies the positivity condition. Moreover, $g$ is extremal since $\sigma_u$ is by Proposition~5.11 in \cite{MS19} and $\rk((V_\Lambda^g)_1)=\rk(V^{\sigma_u}_1)=\rk(V_1)=\rk((V_\Lambda^{\orb(g)})_1)$ since $\sigma_u|_{V_1}$ is inner. Hence, $g$ is a \gdh{}.
\end{proof}
We remark that \autoref{thm:main1} is similar to Theorem~6.5 in \cite{MS19} with the crucial difference that here we assume the existence of the \voa{} $V$ with a certain weight-one Lie algebra $\g$ while in \cite{MS19} the existence is proved for each of the $71$ Lie algebras on Schellekens' list by explicitly specifying a \gdh{} in $\Aut(V_\Lambda)$ such that $(V_\Lambda^{\orb(g)})_1\cong\g$.


\section{New Proof of Schellekens' List}\label{sec:newproof}
In this section we give a simpler proof of Schellekens' list of possible weight-one Lie algebras of \strathol{} \voa{}s of central charge $24$. The proof uses \autoref{thm:main1}, namely that all \strathol{} \voa{}s $V$ of central charge $24$ with non-zero weight-one space $V_1$ can be obtained as orbifold constructions from the Leech lattice \voa{} $V_\Lambda$ associated with \gdh{}s in $\Aut(V_\Lambda)$.

First, based on this result, we derive a couple of simple identities that must be satisfied by any such weight-one Lie algebra. For convenience we include equation~\eqref{eq:221}, which was already used in the proof of \autoref{thm:main1}.
\begin{lem}\label{lem:conditions}
Let $V$ be a \strathol{} \voa{} of central charge $24$ with $V_1$ semisimple and affine structure $\g_{1,k_1}\ldots\g_{r,k_r}$. Then
\begin{enumerate}
\item $h_i^\vee/k_i=(\dim(V_1)-24)/24$
\end{enumerate}
for all $i=1,\ldots,r$, and there exists $\nu\in\O(\Lambda)$ such that
\begin{enumerate}[resume]\itemsep1.5mm
\item $\rk(\Lambda^\nu)=\rk(V_1)$,\label{item:cond2}
\item $|\nu|\,\big|\,\lcm(\{l_ih_i^\vee\}_{i=1}^r)$,\label{item:cond3}
\item $1/(1-\rho_\nu)=\lcm(\{l_ik_i\}_{i=1}^r)$.\label{item:cond4}
\end{enumerate}
\end{lem}
Recall that $\rho_\nu$ denotes the vacuum anomaly of $\nu$ and only depends on the cycle shape of $\nu$. Also recall the lacing numbers $l_i$ and dual Coxeter numbers $h_i^\vee$.
\begin{proof}
By \autoref{prop:premain1} and \autoref{thm:main1} there must be a \gdh{} $g\in\Aut(V_\Lambda)$ of order $n=|\sigma_u|$ such that $V_\Lambda^{\orb(g)}\cong V$. As explained in \autoref{sec:autleech}, $g$ projects to an automorphism $\nu\in\O(\Lambda)$ of order dividing $n$. The order $n=\lcm(\{l_ih_i^\vee\}_{i=1}^r)$ of $\sigma_u$ was determined in \autoref{prop:prepremain1_1}. This implies \eqref{item:cond3}.

Since the Leech lattice $\Lambda$ has no roots, $(V_\Lambda^g)_1=\{h(-1)\otimes\ee_0\,|\,h\in\pi_\nu(\Lambda_\C)\}$, which is an abelian Lie algebra of rank equal to $\dim(\pi_\nu(\Lambda_\C))=\rk(\Lambda^\nu)$, the rank of the fixed-point lattice. The rank condition in the definition of \gdh{}s states that $\rk(V_1^g)=\rk(V_1^{\orb(g)})$ so that \eqref{item:cond2} holds.

A \gdh{} $g$ is also extremal, i.e. the upper bound in the dimension formula is attained. This bound simplifies for the Leech lattice \voa{} $V_\Lambda$ because $(V_\Lambda)_1$ is isomorphic to $\Lambda_\C$, on which $g$ acts as $\nu$. Suppose $\nu$ has order~$m$ and cycle shape $\prod_{t\mid m}t^{b_t}$. Then
\begin{equation*}
\dim((V_\Lambda^{\orb(g)})_1)-24=\sum_{d\mid n}c_n(d)\dim(V_1^{g^d})=n\sum_{t\mid m}\frac{b_t}{t}=24n(1-\rho_\nu)
\end{equation*}
with vacuum anomaly $\rho_\nu$. In the second step we used exactly the same argument as in the proof of \autoref{thm:dimstrange}. With equation~\eqref{eq:221} this implies
\begin{equation*}
\frac{h_i^\vee}{k_i}=\frac{\dim((V_\Lambda^{\orb(g)})_1)-24}{24}=n(1-\rho_\nu)
\end{equation*}
for all $i=1,\ldots,r$, which entails item~\eqref{item:cond4}.
\end{proof}

It is straightforward to list all solutions, i.e.\ pairs of affine structures and automorphisms of the Leech lattice $\Lambda$, to the equations in \autoref{lem:conditions}.
\begin{prop}\label{prop:13}
There are exactly $82$ pairs of affine structures and conjugacy classes in $\O(\Lambda)$ satisfying the four equations in \autoref{lem:conditions}. These are the $69$ semisimple cases on Schellekens' list \cite{Sch93}, each with one of $11$ conjugacy classes in $\O(\Lambda)$ (see \autoref{table:69}), plus the $13$ spurious cases listed in \autoref{table:13}.
\end{prop}
\begin{table}[ht]
\[
\begin{array}{llcr}
\nu\in\O(\Lambda) & \rho_\nu & n & \text{Aff.\ Struct.} \\\hline\hline
6^4          & 35/36 &  6 & D_{4,36}                \\\hline
\multirow{2}{*}{$4^6$} & \multirow{2}{*}{$15/16$}
                     &  4 & A_{3,16}^2              \\
             &       &  8 & C_{3,8}A_{3,8}          \\\hline
\multirow{3}{*}{$3^8$} & \multirow{3}{*}{$8/9$}
                     &  3 & A_{2,9}^4               \\
             &       &  6 & D_{4,9}A_{1,3}^4        \\
             &       & 12 & G_{2,3}^4               \\\hline
\multirow{2}{*}{$2^44^4$} & \multirow{2}{*}{$7/8$}
                     &  4 & A_{3,8}^2A_{1,4}^2      \\
             &       &  8 & C_{3,4}^2A_{1,2}^2      \\\hline
1^22^23^26^2 & 5/6   &  6 & D_{4,6}B_{2,3}^2        \\\hline
\multirow{4}{*}{$2^{12}$} & \multirow{4}{*}{$3/4$}
                     &  4 & A_{3,4}^2A_{1,2}^6      \\
             &       &  8 & D_{5,4}A_{3,2}A_{1,1}^4 \\
             &       &  8 & C_{3,2}^3A_{1,1}^3      \\
             &       &  8 & C_{3,2}^2A_{3,2}^2
\end{array}
\]
\caption{$13$ spurious cases in \autoref{prop:13}.}
\label{table:13}
\end{table}
Note that there is no affine structure that appears in more than one pair.
\begin{proof}
The $221$ affine structures satisfying equation~\eqref{eq:221} are listed in \autoref{table:221}. The $167$ conjugacy classes of $\O(\Lambda)=\Co_0$ are well-known (see, e.g., \cite{CCNPW85}). It is now a purely combinatorial problem to find the pairs consisting of an affine structure $\g_{1,k_1}\ldots\g_{r,k_r}$ and a conjugacy class $\nu\in\O(\Lambda)$ satisfying items \eqref{item:cond2} to \eqref{item:cond4}.

For example, since $\rk(V_1)>0$, item~\eqref{item:cond2} implies that $\nu$ has non-trivial fixed-point lattice, leaving only $72$ conjugacy classes in $\O(\Lambda)$, and item~\eqref{item:cond4} implies that $\rho_\nu<1$, which further reduces the possible conjugacy classes to $50$. Also, since $\rk(\Lambda^\nu)$ is always even, which is a special property of the Leech lattice $\Lambda$, all affine structures with $\rk(V_1)$ odd are eliminated. Moreover, as we mentioned in the proof of \autoref{prop:prepremain1_1}, if the order $n$ is odd, then $V_1$ can only contain simple ideals $A_l$ with $l$ even.
\end{proof}

In the following we rule out these $13$ spurious cases. So far, we have only considered the compatibility of the affine structure with the projection of the \gdh{} to $\O(\Lambda)$. To use the full strength of \autoref{thm:main1} we classify (see \autoref{prop:autleech} and \autoref{prop:autleech2}) all possible \gdh{}s in $\Aut(V_\Lambda)$ with given projection to $\O(\Lambda)$ and order $n$ and show that none exist whose corresponding orbifold constructions have the affine structures in \autoref{table:13}.
\begin{thm}\label{thm:main2}
Let $V$ be a \strathol{} \voa{} of central charge $24$ with $V_1\neq\{0\}$. Then $V_1$ is isomorphic to one of the $70$ non-zero Lie algebras in Table~1 of \cite{Sch93}, each uniquely specifying the affine structure of $\langle V_1\rangle$, and $V$ is isomorphic to $V_\Lambda^{\orb(g)}$ for a \gdh{} $g\in\Aut(V_\Lambda)$, projecting to one of $11$ conjugacy classes in $\O(\Lambda)=\Co_0$.
\end{thm}
In \autoref{table:69} we list the $69$ semisimple cases together with some properties of the corresponding \gdh{}s in $\Aut(V_\Lambda)$. The $11$ conjugacy classes in $\O(\Lambda)$ are exactly those given in Table~4 (and Tables 5 to 15) in \cite{Hoe17} where they arise in a very different approach.

In \cite{MS19} a \gdh{} for each weight-one Lie algebra is explicitly listed, thus providing a uniform proof of the existence of \strathol{} \voa{}s of central charge $24$ for all the $71$ Lie algebras on Schellekens' list.
\begin{proof}
Given \autoref{lem:conditions} and \autoref{prop:13} it remains to eliminate the $13$ potential spurious cases in \autoref{table:13}. Using \autoref{prop:autleech} and \autoref{prop:autleech2} we classify the conjugacy classes of automorphisms $g\in\Aut(V_\Lambda)$ of order $n$ projecting to the listed conjugacy class $\nu\in\O(\Lambda)$. Then we show that either none of these classes $g$ is a \gdh{} or that it is a \gdh{} corresponding to one of the $69$ semisimple Lie algebras on Schellekens' list.

The computationally most challenging part is the determination of the orbits of the action of $C_{\O(\Lambda)}(\nu)$ on $(\Lambda^\nu/n)/\pi_{\nu}(\Lambda)$, which has $n^{\rk(\Lambda^\nu)}/|(\Lambda^\nu)'/\Lambda^\nu|$ elements. This is performed on the computer using \texttt{Magma} \cite{Magma}. For the cases considered, the computations take between seconds and a couple of minutes on a standard desktop computer.

In order to prove that an automorphism in $\Aut(V_\Lambda)$ is not a \gdh{} (see \autoref{def:gendeephole}) we show that it does not have type~$0$ (by computing the conformal weight $\rho(V_\Lambda(g))$, see \autoref{sec:autleech}), is not extremal or does not satisfy the rank condition~\eqref{item:gendeephole3}.

Sufficient and necessary criteria for extremality can be derived from the dimension formula (see, e.g., Proposition~5.9 in \cite{MS19}). For instance, if $g$ is extremal, then $\rho(V_\Lambda(g^i))\geq1$ for all $i\in\Z_n$ with $(i,n)=1$.

To prove that an automorphism does not satisfy the rank condition we show that the necessary criterion in Remark~3.8 of \cite{HM20} is not satisfied. A sufficient criterion is given in Proposition~3.7 of \cite{HM20}.

For two of the following automorphisms these criteria will be inconclusive. We then assume that the rank condition is fulfilled, which is equivalent to $(V_\Lambda^g)_1$ being a Cartan subalgebra of the semisimple Lie algebra $(V_\Lambda^{\orb(g)})_1$, so that the action of the zero modes of $(V_\Lambda^g)_1$ on $(V_\Lambda^{\orb(g)})_1$ yields the root space decomposition of $(V_\Lambda^{\orb(g)})_1$. The action of $(V_\Lambda^g)_1$ on the irreducible $g^i$-twisted $V_\Lambda$-module $V_\Lambda(g^i)$, $i\in\Z_n$, follows directly from the explicit construction in \cite{DL96,BK04} (see \autoref{sec:autleech}). We then consider the inclusions
\begin{equation*}
(V_\Lambda^g)_1\oplus\!\!\bigoplus_{\substack{i\in\Z_n\\(i,n)=1}}\!\!(V_\Lambda(g^i))_1\subseteq(V_\Lambda^{\orb(g)})_1\subseteq(V_\Lambda^g)_1\oplus\bigoplus_{\substack{i\in\Z_n\\i\neq0}}(V_\Lambda(g^i))_1,
\end{equation*}
which allow us to compute a sub- and superset of the root system of $(V_\Lambda^{\orb(g)})_1$, including the lengths with respect to the unique non-degenerate, invariant bilinear form $\langle\cdot,\cdot\rangle$ on $V_\Lambda^{\orb(g)}$ normalised such that $\langle\vac,\vac\rangle=-1$. This leads to a contradiction if the sub- or superset is not compatible with the affine structure.

We compute that in $\Aut(V_\Lambda)$:
\begin{enumerate}[leftmargin=*]
\item There is exactly one conjugacy class of order~$6$ projecting to the conjugacy class in $\O(\Lambda)$ with cycle shape $6^4$ but its type is not $0$.
\item There is exactly one conjugacy class of order~$4$ projecting to the conjugacy class in $\O(\Lambda)$ with cycle shape $4^6$ but its type is not $0$.
\item There are exactly six conjugacy classes of order~$8$ projecting to the conjugacy class in $\O(\Lambda)$ with cycle shape $4^6$. Five of these do not have type~$0$, and one has type~$0$ but does not satisfy the rank condition.
\item There is exactly one conjugacy class of order~$3$ projecting to the conjugacy class in $\O(\Lambda)$ with cycle shape $3^8$ but its type is not $0$.
\item There are exactly two conjugacy classes of order~$6$ projecting to the conjugacy class in $\O(\Lambda)$ with cycle shape $3^8$ but their types are not $0$.
\item There are exactly eight conjugacy classes of order~$12$ projecting to the conjugacy class in $\O(\Lambda)$ with cycle shape $3^8$ but their types are not $0$.
\item There are exactly three conjugacy classes of order~$4$ projecting to the conjugacy class in $\O(\Lambda)$ with cycle shape $2^44^4$. Two of these do not have type~$0$, and one has type~$0$ but does not satisfy the rank condition.
\item There are exactly $15$ conjugacy classes of order~$8$ projecting to the conjugacy class in $\O(\Lambda)$ with cycle shape $2^44^4$. $13$ of these do not have type~$0$, and two have type~$0$ but do not satisfy the rank condition.
\item There are exactly $13$ conjugacy classes of order~$6$ projecting to the conjugacy class in $\O(\Lambda)$ with cycle shape $1^22^23^26^2$. Nine of these do not have type~$0$, two have type~$0$ but do not satisfy the rank condition, and the two remaining ones are discussed below.\label{item:open}
\item There are exactly seven conjugacy classes of order~$4$ projecting to the conjugacy class in $\O(\Lambda)$ with cycle shape $2^{12}$. Five of these do not have type~$0$, and two have type~$0$ but do not satisfy the rank condition.
\item There are exactly $56$ conjugacy classes of order~$8$ projecting to the conjugacy class in $\O(\Lambda)$ with cycle shape $2^{12}$. $50$ of these do not have type~$0$, five have type~$0$ but do not satisfy the rank condition, and one further case has type~$0$ but is not extremal since $\rho(V_\Lambda(g))=7/8<1$.
\end{enumerate}
It remains to study the two remaining conjugacy classes of order~$6$ and type~$0$ projecting to the conjugacy class in $\O(\Lambda)$ with cycle shape $1^22^23^26^2$ from item~\eqref{item:open}.

The first conjugacy class satisfies the rank condition and is extremal. Hence, it is a \gdh{}. However, the action of $(V_\Lambda^g)_1$ on the twisted modules is not compatible with the affine structure $(V_\Lambda^{\orb(g)})_1\cong D_{4,6}B_{2,3}^2$. (Specifically, the action on $V_\Lambda(g)_1$ defines simple roots for an (affine) root system of type $\tilde{A}_5A_1$.) By exclusion, $(V_\Lambda^{\orb(g)})_1\cong A_{5,6}B_{2,3}A_{1,2}$, corresponding to case $8$ on Schellekens' list.

The second conjugacy class is also extremal. With a similar argument as before we see that $V_\Lambda^{\orb(g)}$ cannot have affine structure $D_{4,6}B_{2,3}^2$ (nor can it have affine structure $A_{5,6}B_{2,3}A_{1,2}$, and so it does not satisfy the rank condition).
\end{proof}

We remark that in this work we only determine the $69$ possible affine structures, i.e.\ the affine \vosa{}s $\langle V_1\rangle\cong\bigotimes_{i=1}^rL_{\hat\g_i}(k_i,0)$, that can occur in a \strathol{} \voa{} of central charge $24$. Schellekens, by contrast, also determined the possible decompositions of $V$ as a $\langle V_1\rangle$-module (see Table~1 in \cite{Sch93}). He obtains that this decomposition is unique up to an outer automorphism.


\section{Towards a Uniform Uniqueness Proof}\label{sec:uniqueness}
So far, the proof of the uniqueness of a \strathol{} \voa{} of central charge $24$ with a given non-zero weight-one space is scattered over many publications \cite{DM04,LS15a,LS19,KLL18,LL20,EMS20b,LS20,LS20b}.

A second potential application of \autoref{thm:main1} is a more uniform proof of this uniqueness statement. In this section we demonstrate this approach for $43$ of the $70$ non-zero weight-one Lie algebras on Schellekens' list. While we believe that this method works in principle for all cases, we are currently restricted by computational limitations.

Another uniform proof of the uniqueness statement is given in \cite{HM20} based on orbifold constructions not just associated with the Leech lattice but with all Niemeier lattices.

In the following we try to classify the \gdh{}s in $\Aut(V_\Lambda)$. If we can show that there is only one conjugacy class $g$ of \gdh{}s in $\Aut(V_\Lambda)$ whose corresponding orbifold construction has a certain weight-one Lie algebra $\g$ on Schellekens' list, then any \strathol{} \voa{} of central charge $24$ with $V_1\cong\g$ must be isomorphic to $V_\Lambda^{\orb(g)}$, proving the uniqueness for this $\g$.

\begin{thm}\label{thm:main3}
Let $V$ be a \strathol{} \voa{} of central charge $24$ with weight-one space $V_1\neq\{0\}$. Then the Lie algebra structure of $V_1$ uniquely determines the \voa{} $V$ up to isomorphism if $V_1$ is one of the $24$ Lie algebras of rank $24$ or corresponding to cases $2,\ldots,14$ and $16,\ldots,21$ on Schellekens' list. Moreover, there is a unique conjugacy class of \gdh{}s in $\Aut(V_\Lambda)$ such that $V_\Lambda^{\orb(g)}\cong V$ in these cases.
\end{thm}
\begin{proof}
We begin with the case of $\rk(V_1)=24$. Then $V\cong V_\Lambda^{\orb(g)}$ for some \gdh{} $g\in\Aut(V_\Lambda)$ projecting to $\nu\in\O(\Lambda)$ with $\rk(\Lambda^\nu)=24$. This means that $\nu=\id$, i.e.\ $g$ is inner and conjugate to $\sigma_h=\e^{-(2\pi\i)h_0}$ for some $h\in\Lambda_\C/\Lambda$. Recall that $g=\id$ is a \gdh{} by convention and this is the unique \gdh{} $g$ such that $V_\Lambda^{\orb(g)}\cong V_\Lambda$. In the following we assume that $g\neq\id$ or equivalently that $V_1$ is semisimple.

By the dimension formula in \cite{MS19}, the extremality of $g$ implies $\rho(V_\Lambda(g))\geq 1$. On the other hand, for $d\in\Lambda_\C$, $\rho(V_\Lambda(\sigma_d))=\min_{\alpha\in\Lambda+d}\langle\alpha,\alpha\rangle/2$. Since the covering radius of the Leech lattice is $\sqrt{2}$ \cite{CPS82}, the conformal weight can be at most $1$ and it is $1$ by definition if and only if $d$ is a \emph{deep hole} of the Leech lattice.

Hence, $h$ must be a deep hole. The deep holes of the Leech lattice $\Lambda$ were classified in \cite{CPS82} and it is shown that there are exactly $23$ orbits of deep holes in $\Lambda_\C/\Lambda$ under the action of $\O(\Lambda)=C_{\O(\Lambda)}(\id)$. This shows that there are at most $24$ conjugacy classes of \gdh{}s in $\Aut(V_\Lambda)$ with $\rk(\Lambda^\nu)=24$, each one except for the identity represented by $\sigma_h$ where $h$ is a representative of a deep hole of $\Lambda$. It is not difficult to see that each of the $23$ deep holes in fact defines a \gdh{} in this way and based on \cite{CS82} it is proved in \cite{MS19} that the orbifold constructions associated with these $23$ \gdh{}s give exactly the \voa{}s associated with the $23$ Niemeier lattices other than the Leech lattice. This proves the uniqueness of the $24$ lattice cases on Schellekens' list, i.e.\ the cases with $\rk(V_1)=24$.

\medskip

For the weight-one Lie algebras of rank less than $24$ we proceed by classifying the corresponding \gdh{}s on the computer using \texttt{Magma} \cite{Magma} similar to the proof of \autoref{thm:main2}. Recall that it was shown in \cite{MS19} that there is at least one \gdh{} corresponding to each case on Schellekens' list.

We compute that in $\Aut(V_\Lambda)$:
\begin{enumerate}[leftmargin=*]
\item There are exactly six conjugacy classes of order~$10$ projecting to the conjugacy class in $\O(\Lambda)$ with cycle shape $2^210^2$. Four of these do not have type~$0$, and one has type~$0$ but does not satisfy the rank condition. The remaining class must be a \gdh{} corresponding to case~$4$ on Schellekens' list.
\item There are exactly four conjugacy classes of order~$6$ projecting to the conjugacy class in $\O(\Lambda)$ with cycle shape $2^36^3$. Two of these do not have type~$0$, and one has type~$0$ but does not satisfy the rank condition. The remaining class must be a \gdh{} corresponding to case~$3$ on Schellekens' list.
\item There are exactly $156$ conjugacy classes of order~$18$ projecting to the conjugacy class in $\O(\Lambda)$ with cycle shape $2^36^3$. $135$ of these do not have type~$0$, and four have type~$0$ but do not satisfy the rank condition. Of the remaining $17$ classes $16$ are not extremal. The remaining class must be a \gdh{} corresponding to case~$14$ on Schellekens' list.
\item There are exactly $25$ conjugacy classes of order~$8$ projecting to the conjugacy class in $\O(\Lambda)$ with cycle shape $1^22^14^18^2$. $21$ of these do not have type~$0$, and three have type~$0$ but do not satisfy the rank condition. The remaining class must be a \gdh{} corresponding to case~$10$ on Schellekens' list.
\item There are exactly eight conjugacy classes of order~$7$ projecting to the conjugacy class in $\O(\Lambda)$ with cycle shape $1^37^3$. Six of these do not have type~$0$, and one has type~$0$ but does not satisfy the rank condition. The remaining class must be a \gdh{} corresponding to case~$11$ on Schellekens' list.
\item As was already discussed in the proof of \autoref{thm:main2}, there are exactly $13$ conjugacy classes of order~$6$ projecting to the conjugacy class in $\O(\Lambda)$ with cycle shape $1^22^23^26^2$. Nine of these do not have type~$0$, and three have type~$0$ but do not satisfy the rank condition. The remaining class must be a \gdh{} corresponding to case~$8$ on Schellekens' list.
\item There are exactly $194$ conjugacy classes of order~$12$ projecting to the conjugacy class in $\O(\Lambda)$ with cycle shape $1^22^23^26^2$. $172$ of these do not have type~$0$, and eight have type~$0$ but do not satisfy the rank condition. Of the remaining $14$ classes $13$ are not extremal. The remaining class must be a \gdh{} corresponding to case~$21$ on Schellekens' list.
\item There are exactly six conjugacy classes of order~$5$ projecting to the conjugacy class in $\O(\Lambda)$ with cycle shape $1^45^4$. Four of these do not have type~$0$, and one has type~$0$ but does not satisfy the rank condition. The remaining class must be a \gdh{} corresponding to case~$9$ on Schellekens' list.
\item There are exactly $54$ conjugacy classes of order~$10$ projecting to the conjugacy class in $\O(\Lambda)$ with cycle shape $1^45^4$. $47$ of these do not have type~$0$, and one has type~$0$ but does not satisfy the rank condition. Of the remaining six classes five are not extremal. The remaining class must be a \gdh{} corresponding to case~$20$ on Schellekens' list.
\item There are exactly nine conjugacy classes of order~$4$ projecting to the conjugacy class in $\O(\Lambda)$ with cycle shape $1^42^24^4$. Six of these do not have type~$0$, and two have type~$0$ but do not satisfy the rank condition. The remaining class must be a \gdh{} corresponding to case~$7$ on Schellekens' list.
\item There are exactly $94$ conjugacy classes of order~$8$ projecting to the conjugacy class in $\O(\Lambda)$ with cycle shape $1^42^24^4$. $82$ of these do not have type~$0$, and four have type~$0$ but do not satisfy the rank condition. Of the remaining eight classes six are not extremal. The two remaining classes must be \gdh{}s corresponding to cases $18$ and $19$ on Schellekens' list.
\item There is exactly one conjugacy class of order~$2$ projecting to the conjugacy class in $\O(\Lambda)$ with cycle shape $2^{12}$. This class must be a \gdh{} corresponding to case~$2$ on Schellekens' list.
\item There are exactly $14$ conjugacy classes of order~$6$ projecting to the conjugacy class in $\O(\Lambda)$ with cycle shape $2^{12}$. Eight of these do not have type~$0$, four have type~$0$ and satisfy the rank condition but are not extremal. The two remaining classes must be \gdh{}s corresponding to cases $12$ and $13$ on Schellekens' list.
\item There are exactly four conjugacy classes of order~$3$ projecting to the conjugacy class in $\O(\Lambda)$ with cycle shape $1^63^6$. Two of these do not have type~$0$, and one has type~$0$ but does not satisfy the rank condition. The remaining class must be a \gdh{} corresponding to case~$6$ on Schellekens' list.
\item There are exactly $24$ conjugacy classes of order~$6$ projecting to the conjugacy class in $\O(\Lambda)$ with cycle shape $1^63^6$. $18$ of these do not have type~$0$, and one has type~$0$ but does not satisfy the rank condition. Of the remaining five classes four are not extremal. The remaining class must be a \gdh{} corresponding to case~$17$ on Schellekens' list.
\item There are exactly three conjugacy classes of order~$2$ projecting to the conjugacy class in $\O(\Lambda)$ with cycle shape $1^82^8$. One of these does not have type~$0$, and one has type~$0$ but does not satisfy the rank condition. The remaining class must be a \gdh{} corresponding to case~$5$ on Schellekens' list.
\item There are exactly $14$ conjugacy classes of order~$4$ projecting to the conjugacy class in $\O(\Lambda)$ with cycle shape $1^82^8$. Nine of these do not have type~$0$, and one has type~$0$ but does not satisfy the rank condition. Of the remaining four classes three are not extremal. The remaining class must be a \gdh{} corresponding to case~$16$ on Schellekens' list.\qedhere
\end{enumerate}
\end{proof}
For the other cases on Schellekens' list determining the conjugacy classes in $\Aut(V_\Lambda)$ and specifically the orbit representatives of the action of $C_{\O(\Lambda)}(\nu)$ on $(\Lambda^\nu/n)/\pi_{\nu}(\Lambda)$ takes prohibitively long on a standard desktop computer.


\newpage
\begin{appendix}
\section{Tables}
\begin{table}[ht]
\[
\begin{array}{ccrrllr}
\nu\in\O(\Lambda) & \text{Rk.} & n & \text{Dim.} & \text{Aff.\ Struct.} & \text{\cite{Hoe17}} & \text{\cite{Sch93}} \\ \hline\hline
\multirow{24}{*}{$1^{24}$} & \multirow{24}{*}{$24$}
  & 46 & 1128 & D_{24,1}                   & A1  & [70] \\ \cline{3-7}
& & 30 &  744 & D_{16,1}E_{8,1}            & A2  & [69] \\ \cline{3-7}
& & 30 &  744 & E_{8,1}^3                  & A3  & [68] \\ \cline{3-7}
& & 25 &  624 & A_{24,1}                   & A4  & [67] \\ \cline{3-7}
& & 22 &  552 & D_{12,1}^2                 & A5  & [66] \\ \cline{3-7}
& & 18 &  456 & A_{17,1}E_{7,1}            & A6  & [65] \\ \cline{3-7}
& & 18 &  456 & D_{10,1}E_{7,1}^2          & A7  & [64] \\ \cline{3-7}
& & 16 &  408 & A_{15,1}D_{9,1}            & A8  & [63] \\ \cline{3-7}
& & 14 &  360 & D_{8,1}^3                  & A9  & [61] \\ \cline{3-7}
& & 13 &  336 & A_{12,1}^2                 & A10 & [60] \\ \cline{3-7}
& & 12 &  312 & A_{11,1}D_{7,1}E_{6,1}     & A11 & [59] \\ \cline{3-7}
& & 12 &  312 & E_{6,1}^4                  & A12 & [58] \\ \cline{3-7}
& & 10 &  264 & A_{9,1}^2D_{6,1}           & A13 & [55] \\ \cline{3-7}
& & 10 &  264 & D_{6,1}^4                  & A14 & [54] \\ \cline{3-7}
& &  9 &  240 & A_{8,1}^3                  & A15 & [51] \\ \cline{3-7}
& &  8 &  216 & A_{7,1}^2D_{5,1}^2         & A16 & [49] \\ \cline{3-7}
& &  7 &  192 & A_{6,1}^4                  & A17 & [46] \\ \cline{3-7}
& &  6 &  168 & A_{5,1}^4D_{4,1}           & A18 & [43] \\ \cline{3-7}
& &  6 &  168 & D_{4,1}^6                  & A19 & [42] \\ \cline{3-7}
& &  5 &  144 & A_{4,1}^6                  & A20 & [37] \\ \cline{3-7}
& &  4 &  120 & A_{3,1}^8                  & A21 & [30] \\ \cline{3-7}
& &  3 &   96 & A_{2,1}^{12}               & A22 & [24] \\ \cline{3-7}
& &  2 &   72 & A_{1,1}^{24}               & A23 & [15] \\ \cline{3-7}
& &  1 &   24 & \C^{24}                    & A24 &  [1] \\ \hline\hline
\multirow{8}{*}{$1^82^8$} & \multirow{8}{*}{$16$}
  & 30 &  384 & E_{8,2}B_{8,1}             & B1  & [62] \\ \cline{3-7}
& & 22 &  288 & C_{10,1}B_{6,1}            & B2  & [56] \\ \cline{3-7}
& & 18 &  240 & C_{8,1}F_{4,1}^2           & B3  & [52] \\ \cline{3-7}
& & 18 &  240 & E_{7,2}B_{5,1}F_{4,1}      & B4  & [53] \\ \cline{3-7}
& & 16 &  216 & D_{9,2}A_{7,1}             & B5  & [50] \\ \cline{3-7}
& & 14 &  192 & D_{8,2}B_{4,1}^2           & B6  & [47] \\ \cline{3-7}
& & 14 &  192 & C_{6,1}^2B_{4,1}           & B7  & [48] \\ \cline{3-7}
& & 12 &  168 & E_{6,2}C_{5,1}A_{5,1}      & B8  & [44]
\end{array}
\]
\caption{The $69$ \strathol{} \voa{}s $V$ of central charge $24$ with $V_1$ semisimple and the corresponding \gdh{}s in $\nu\in\Aut(V_\Lambda)$ (continued on next page).}
\label{table:69}
\end{table}
\addtocounter{table}{-1}
\begin{table}[ht]
\[
\begin{array}{ccrrllr}
\nu\in\O(\Lambda)& \text{Rk.} & n & \text{Dim.} & \text{Aff.\ Struct.} & \text{\cite{Hoe17}} & \text{\cite{Sch93}} \\ \hline\hline
\multirow{9}{*}{$1^82^8$} & \multirow{9}{*}{$16$}
  & 10 &  144 & A_{9,2}A_{4,1}B_{3,1}      & B9  & [40] \\ \cline{3-7}
& & 10 &  144 & D_{6,2}C_{4,1}B_{3,1}^2    & B10 & [39] \\ \cline{3-7}
& & 10 &  144 & C_{4,1}^4                  & B11 & [38] \\ \cline{3-7}
& &  8 &  120 & A_{7,2}C_{3,1}^2A_{3,1}    & B12 & [33] \\ \cline{3-7}
& &  8 &  120 & D_{5,2}^2A_{3,1}^2         & B13 & [31] \\ \cline{3-7}
& &  6 &   96 & A_{5,2}^2B_{2,1}A_{2,1}^2  & B14 & [26] \\ \cline{3-7}
& &  6 &   96 & D_{4,2}^2B_{2,1}^4         & B15 & [25] \\ \cline{3-7}
& &  4 &   72 & A_{3,2}^4A_{1,1}^4         & B16 & [16] \\ \cline{3-7}
& &  2 &   48 & A_{1,2}^{16}               & B17 &  [5] \\ \hline\hline
\multirow{6}{*}{$1^63^6$} & \multirow{6}{*}{$12$}
  & 18 &  168 & E_{7,3}A_{5,1}             & C1  & [45] \\ \cline{3-7}
& & 12 &  120 & D_{7,3}A_{3,1}G_{2,1}      & C2  & [34] \\ \cline{3-7}
& & 12 &  120 & E_{6,3}G_{2,1}^3           & C3  & [32] \\ \cline{3-7}
& &  9 &   96 & A_{8,3}A_{2,1}^2           & C4  & [27] \\ \cline{3-7}
& &  6 &   72 & A_{5,3}D_{4,3}A_{1,1}^3    & C5  & [17] \\ \cline{3-7}
& &  3 &   48 & A_{2,3}^{6}                & C6  &  [6] \\ \hline\hline
\multirow{9}{*}{$2^{12}$} & \multirow{9}{*}{$12$}
  & 46 &  300 & B_{12,2}                   & D1a & [57] \\
& & 22 &  156 & B_{6,2}^2                  & D1b & [41] \\
& & 14 &  108 & B_{4,2}^3                  & D1c & [29] \\
& & 10 &   84 & B_{3,2}^4                  & D1d & [23] \\
& &  6 &   60 & B_{2,2}^6                  & D1e & [12] \\
& &  2 &   36 & A_{1,4}^{12}               & D1f &  [2] \\ \cline{3-7}
& & 18 &  132 & A_{8,2}F_{4,2}             & D2a & [36] \\
& & 10 &   84 & C_{4,2}A_{4,2}^2           & D2b & [22] \\
& &  6 &   60 & D_{4,4}A_{2,2}^4           & D2c & [13] \\ \hline\hline
\multirow{5}{*}{$1^42^24^4$} & \multirow{5}{*}{$10$}
  & 16 &  120 & C_{7,2}A_{3,1}             & E1  & [35] \\ \cline{3-7}
& & 12 &   96 & E_{6,4}B_{2,1}A_{2,1}      & E2  & [28] \\ \cline{3-7}
& &  8 &   72 & A_{7,4}A_{1,1}^3           & E3  & [18] \\ \cline{3-7}
& &  8 &   72 & D_{5,4}C_{3,2}A_{1,1}^2    & E4  & [19] \\ \cline{3-7}
& &  4 &   48 & A_{3,4}^3A_{1,2}           & E5  &  [7] \\ \hline\hline
\multirow{2}{*}{$1^45^4$} & \multirow{2}{*}{$8$}
  & 10 &   72 & D_{6,5}A_{1,1}^2           & F1  & [20] \\ \cline{3-7}
& &  5 &   48 & A_{4,5}^2                  & F2  &  [9] \\ \hline\hline
\multirow{2}{*}{$1^22^23^26^2$} & \multirow{2}{*}{$8$}
  & 12 &   72 & C_{5,3}G_{2,2}A_{1,1}      & G1  & [21] \\ \cline{3-7}
& &  6 &   48 & A_{5,6}B_{2,3}A_{1,2}      & G2  &  [8] \\ \hline\hline
\multirow{1}{*}{$1^37^3$} & \multirow{1}{*}{$6$}
  &  7 &   48 &  A_{6,7}                   & H1  & [11] \\ \hline\hline
\multirow{1}{*}{$1^22^14^18^2$} & \multirow{1}{*}{$6$}
  &  8 &   48 & D_{5,8}A_{1,2}             & I1  & [10] \\ \hline\hline
\multirow{2}{*}{$2^36^3$} & \multirow{2}{*}{$6$}
  & 18 &   60 & F_{4,6}A_{2,2}             & J1a & [14] \\
& &  6 &   36 & D_{4,12}A_{2,6}            & J1b &  [3] \\ \hline\hline
\multirow{1}{*}{$2^210^2$} & \multirow{1}{*}{$4$}
  & 10 &   36 & C_{4,10}                   & K1  &  [4]
\end{array}
\]
\caption{(continued)}
\end{table}

\begin{table}[ht]
\begin{minipage}[t]{0.303\textwidth}
\[
\begin{array}{r|l}
\text{Dim.} & \text{Aff.\ Struct.}\\\hline\hline
25 & A_{3,96}B_{2,72} \\
25 & G_{2,96}A_{2,72}A_{1,48} \\
25 & B_{2,72}A_{1,48}^5 \\
25 & A_{2,72}^2A_{1,48}^3 \\\hline
26 & A_{3,48}A_{2,36}A_{1,24} \\
26 & G_{2,48}A_{1,24}^4 \\
26 & B_{2,36}^2A_{1,24}^2 \\
26 & B_{2,36}A_{2,36}^2 \\
26 & A_{2,36}A_{1,24}^6 \\\hline
27 & A_{4,40}A_{1,16} \\
27 & B_{3,40}A_{1,16}^2 \\
27 & C_{3,32}A_{1,16}^2 \\
27 & A_{3,32}A_{1,16}^4 \\
27 & G_{2,32}B_{2,24}A_{1,16} \\
27 & B_{2,24}A_{2,24}A_{1,16}^3 \\
27 & A_{2,24}^3A_{1,16} \\
27 & A_{1,16}^9 \\\hline
28 & D_{4,36} \\
28 & A_{3,24}B_{2,18}A_{1,12} \\
28 & G_{2,24}^2 \\
28 & G_{2,24}A_{2,18}A_{1,12}^2 \\
28 & B_{2,18}^2A_{2,18} \\
28 & B_{2,18}A_{1,12}^6 \\
28 & A_{2,18}^2A_{1,12}^4 \\\hline
30 & A_{4,20}A_{1,8}^2 \\
30 & B_{3,20}A_{1,8}^3 \\
30 & C_{3,16}A_{1,8}^3 \\
30 & A_{3,16}^2 \\
30 & A_{3,16}A_{1,8}^5 \\
30 & G_{2,16}B_{2,12}A_{1,8}^2 \\
30 & G_{2,16}A_{2,12}^2 \\
30 & B_{2,12}^3 \\
30 & B_{2,12}A_{2,12}A_{1,8}^4 \\
30 & A_{2,12}^3A_{1,8}^2 \\
30 & A_{1,8}^{10} \\\hline
32 & A_{4,15}A_{2,9} \\
32 & B_{3,15}A_{2,9}A_{1,6} \\
32 & C_{3,12}A_{2,9}A_{1,6} \\
32 & A_{3,12}G_{2,12}A_{1,6} \\
\end{array}
\]
\end{minipage}
\begin{minipage}[t]{0.34\textwidth}
\[
\begin{array}{r|l}
\text{Dim.} & \text{Aff.\ Struct.}\\\hline\hline
32 & A_{3,12}A_{2,9}A_{1,6}^3 \\
32 & G_{2,12}B_{2,9}A_{2,9} \\
32 & G_{2,12}A_{1,6}^6 \\
32 & B_{2,9}^2A_{1,6}^4 \\
32 & B_{2,9}A_{2,9}^2A_{1,6}^2 \\
32 & A_{2,9}^4 \\
32 & A_{2,9}A_{1,6}^8 \\\hline
36 & B_{4,14} \\
36 & C_{4,10} \\
36 & D_{4,12}A_{2,6} \\
36 & A_{4,10}A_{1,4}^4 \\
36 & B_{3,10}A_{3,8} \\
36 & C_{3,8}A_{3,8} \\
36 & B_{3,10}A_{1,4}^5 \\
36 & C_{3,8}A_{1,4}^5 \\
36 & A_{3,8}^2A_{1,4}^2 \\
36 & A_{3,8}B_{2,6}A_{2,6}A_{1,4} \\
36 & A_{3,8}A_{1,4}^7 \\
36 & G_{2,8}^2A_{2,6} \\
36 & G_{2,8}B_{2,6}A_{1,4}^4 \\
36 & G_{2,8}A_{2,6}^2A_{1,4}^2 \\
36 & B_{2,6}^3A_{1,4}^2 \\
36 & B_{2,6}^2A_{2,6}^2 \\
36 & B_{2,6}A_{2,6}A_{1,4}^6 \\
36 & A_{2,6}^3A_{1,4}^4 \\
36 & A_{1,4}^{12} \\\hline
40 & D_{4,9}A_{1,3}^4 \\
40 & G_{2,6}^2A_{1,3}^4 \\\hline
42 & B_{2,4}A_{2,4}^4 \\\hline
48 & A_{6,7} \\
48 & D_{5,8}A_{1,2} \\
48 & B_{4,7}A_{1,2}^4 \\
48 & C_{4,5}A_{1,2}^4 \\
48 & A_{5,6}B_{2,3}A_{1,2} \\
48 & D_{4,6}G_{2,4}A_{1,2}^2 \\
48 & D_{4,6}B_{2,3}^2 \\
48 & D_{4,6}A_{2,3}A_{1,2}^4 \\
48 & A_{4,5}^2 \\
48 & A_{4,5}B_{3,5}A_{1,2} \\
\end{array}
\]
\end{minipage}
\begin{minipage}[t]{0.32\textwidth}
\[
\begin{array}{r|l}
\text{Dim.} & \text{Aff.\ Struct.}\\\hline\hline
48 & A_{4,5}C_{3,4}A_{1,2} \\
48 & A_{4,5}A_{3,4}A_{1,2}^3 \\
48 & A_{4,5}G_{2,4}B_{2,3} \\
48 & A_{4,5}B_{2,3}A_{2,3}A_{1,2}^2 \\
48 & A_{4,5}A_{2,3}^3 \\
48 & A_{4,5}A_{1,2}^8 \\
48 & B_{3,5}^2A_{1,2}^2 \\
48 & B_{3,5}C_{3,4}A_{1,2}^2 \\
48 & C_{3,4}^2A_{1,2}^2 \\
48 & B_{3,5}A_{3,4}A_{1,2}^4 \\
48 & C_{3,4}A_{3,4}A_{1,2}^4 \\
48 & B_{3,5}G_{2,4}B_{2,3}A_{1,2} \\
48 & C_{3,4}G_{2,4}B_{2,3}A_{1,2} \\
48 & B_{3,5}B_{2,3}A_{2,3}A_{1,2}^3 \\
48 & C_{3,4}B_{2,3}A_{2,3}A_{1,2}^3 \\
48 & B_{3,5}A_{2,3}^3A_{1,2} \\
48 & C_{3,4}A_{2,3}^3A_{1,2} \\
48 & B_{3,5}A_{1,2}^9 \\
48 & C_{3,4}A_{1,2}^9 \\
48 & A_{3,4}^3A_{1,2} \\
48 & A_{3,4}^2B_{2,3}A_{2,3} \\
48 & A_{3,4}^2A_{1,2}^6 \\
48 & A_{3,4}G_{2,4}B_{2,3}A_{1,2}^3 \\
48 & A_{3,4}G_{2,4}A_{2,3}^2A_{1,2} \\
48 & A_{3,4}B_{2,3}^3A_{1,2} \\
48 & A_{3,4}B_{2,3}A_{2,3}A_{1,2}^5 \\
48 & A_{3,4}A_{2,3}^3A_{1,2}^3 \\
48 & A_{3,4}A_{1,2}^{11} \\
48 & G_{2,4}^3A_{1,2}^2 \\
48 & G_{2,4}^2B_{2,3}^2 \\
48 & G_{2,4}^2A_{2,3}A_{1,2}^4 \\
48 & G_{2,4}B_{2,3}^2A_{2,3}A_{1,2}^2 \\
48 & G_{2,4}B_{2,3}A_{2,3}^3 \\
48 & G_{2,4}B_{2,3}A_{1,2}^8 \\
48 & G_{2,4}A_{2,3}^2A_{1,2}^6 \\
48 & B_{2,3}^4A_{2,3} \\
48 & B_{2,3}^3A_{1,2}^6 \\
48 & B_{2,3}^2A_{2,3}^2A_{1,2}^4 \\
48 & B_{2,3}A_{2,3}^4A_{1,2}^2 \\
\end{array}
\]
\end{minipage}
\bigskip
\caption{The 221 solutions of equation~\eqref{eq:221} (continued on next page).}
\label{table:221}
\end{table}
\addtocounter{table}{-1}
\begin{table}[ht]
\begin{minipage}[t]{0.36\textwidth}
\[
\begin{array}{r|l}
\text{Dim.} & \text{Aff.\ Struct.}\\\hline\hline
48 & B_{2,3}A_{2,3}A_{1,2}^{10} \\
48 & A_{2,3}^6 \\
48 & A_{2,3}^3A_{1,2}^8 \\
48 & A_{1,2}^{16} \\\hline
56 & C_{3,3}^2G_{2,3} \\
56 & G_{2,3}^4 \\\hline
60 & F_{4,6}A_{2,2} \\
60 & D_{4,4}A_{2,2}^4 \\
60 & B_{2,2}^6 \\
60 & B_{2,2}^2A_{2,2}^5 \\\hline
72 & D_{6,5}A_{1,1}^2 \\
72 & A_{7,4}A_{1,1}^3 \\
72 & C_{5,3}G_{2,2}A_{1,1} \\
72 & D_{5,4}C_{3,2}A_{1,1}^2 \\
72 & D_{5,4}A_{3,2}A_{1,1}^4 \\
72 & D_{5,4}A_{1,1}^9 \\
72 & A_{5,3}D_{4,3}A_{1,1}^3 \\
72 & A_{5,3}G_{2,2}^2A_{1,1}^3 \\
72 & D_{4,3}C_{3,2}G_{2,2}A_{1,1}^3 \\
72 & D_{4,3}A_{3,2}^2G_{2,2} \\
72 & D_{4,3}A_{3,2}G_{2,2}A_{1,1}^5 \\
72 & D_{4,3}G_{2,2}A_{1,1}^{10} \\
72 & C_{3,2}^3A_{1,1}^3 \\
72 & C_{3,2}^2A_{3,2}^2 \\
72 & C_{3,2}^2A_{3,2}A_{1,1}^5 \\
72 & C_{3,2}^2A_{1,1}^{10} \\
72 & C_{3,2}A_{3,2}^3A_{1,1}^2 \\
72 & C_{3,2}A_{3,2}^2A_{1,1}^7 \\
72 & C_{3,2}A_{3,2}A_{1,1}^{12} \\
72 & C_{3,2}G_{2,2}^3A_{1,1}^3 \\
72 & C_{3,2}A_{1,1}^{17} \\
72 & A_{3,2}^4A_{1,1}^4 \\
72 & A_{3,2}^3A_{1,1}^9 \\
72 & A_{3,2}^2G_{2,2}^3 \\
72 & A_{3,2}^2A_{1,1}^{14} \\
72 & A_{3,2}G_{2,2}^3A_{1,1}^5 \\
72 & A_{3,2}A_{1,1}^{19} \\
72 & G_{2,2}^3A_{1,1}^{10} \\
72 & A_{1,1}^{24} \\
\end{array}
\]
\end{minipage}
\begin{minipage}[t]{0.30\textwidth}
\[
\begin{array}{r|l}
\text{Dim.} & \text{Aff.\ Struct.}\\\hline\hline
84 & C_{4,2}A_{4,2}^2 \\
84 & B_{3,2}^4 \\\hline
96 & A_{8,3}A_{2,1}^2 \\
96 & E_{6,4}B_{2,1}A_{2,1} \\
96 & F_{4,3}D_{4,2}A_{2,1}^2 \\
96 & F_{4,3}B_{2,1}^2A_{2,1}^3 \\
96 & A_{5,2}^2B_{2,1}A_{2,1}^2 \\
96 & D_{4,2}^2B_{2,1}^3 \\
96 & D_{4,2}^2A_{2,1}^5 \\
96 & D_{4,2}B_{2,1}^6A_{2,1} \\
96 & D_{4,2}B_{2,1}^2A_{2,1}^6 \\
96 & B_{2,1}^8A_{2,1}^2 \\
96 & B_{2,1}^4A_{2,1}^7 \\
96 & A_{2,1}^{12} \\\hline
108 & B_{4,2}^3 \\\hline
120 & C_{7,2}A_{3,1} \\
120 & D_{7,3}A_{3,1}G_{2,1} \\
120 & E_{6,3}C_{3,1}^2 \\
120 & E_{6,3}G_{2,1}^3 \\
120 & A_{7,2}C_{3,1}^2A_{3,1} \\
120 & A_{7,2}A_{3,1}G_{2,1}^3 \\
120 & D_{5,2}^2A_{3,1}^2 \\
120 & D_{5,2}A_{3,1}^5 \\
120 & C_{3,1}^5A_{3,1} \\
120 & C_{3,1}^3A_{3,1}G_{2,1}^3 \\
120 & C_{3,1}A_{3,1}G_{2,1}^6 \\
120 & A_{3,1}^8 \\\hline
132 & A_{8,2}F_{4,2} \\\hline
144 & A_{9,2}A_{4,1}B_{3,1} \\
144 & D_{6,2}C_{4,1}B_{3,1}^2 \\
144 & C_{4,1}^4 \\
144 & C_{4,1}^2A_{4,1}^3 \\
144 & C_{4,1}A_{4,1}B_{3,1}^4 \\
144 & A_{4,1}^6 \\\hline
156 & B_{6,2}^2 \\\hline
168 & E_{7,3}A_{5,1} \\
168 & E_{6,2}C_{5,1}A_{5,1} \\
168 & A_{5,1}^4D_{4,1} \\
168 & D_{4,1}^6 \\
\end{array}
\]
\end{minipage}
\begin{minipage}[t]{0.32\textwidth}
\[
\begin{array}{r|l}
\text{Dim.} & \text{Aff.\ Struct.}\\\hline\hline
192 & D_{8,2}B_{4,1}^2 \\
192 & C_{6,1}^2B_{4,1} \\
192 & A_{6,1}^4 \\
192 & A_{6,1}B_{4,1}^4 \\\hline
216 & D_{9,2}A_{7,1} \\
216 & A_{7,1}^2D_{5,1}^2 \\\hline
240 & C_{8,1}F_{4,1}^2 \\
240 & E_{7,2}B_{5,1}F_{4,1} \\
240 & A_{8,1}^3 \\\hline
264 & A_{9,1}^2D_{6,1} \\
264 & D_{6,1}^4 \\\hline
288 & C_{10,1}B_{6,1} \\\hline
300 & B_{12,2} \\\hline
312 & A_{11,1}D_{7,1}E_{6,1} \\
312 & E_{6,1}^4 \\\hline
336 & A_{12,1}^2 \\\hline
360 & D_{8,1}^3 \\\hline
384 & E_{8,2}B_{8,1} \\\hline
408 & A_{15,1}D_{9,1} \\\hline
456 & A_{17,1}E_{7,1} \\
456 & D_{10,1}E_{7,1}^2 \\\hline
552 & D_{12,1}^2 \\\hline
624 & A_{24,1} \\\hline
744 & D_{16,1}E_{8,1} \\
744 & E_{8,1}^3 \\\hline
1128 & D_{24,1}
\end{array}
\]
\end{minipage}
\bigskip
\caption{(continued)}
\end{table}
\end{appendix}

\FloatBarrier


\bibliographystyle{alpha_noseriescomma}
\bibliography{quellen}{}

\end{document}